\documentclass[11pt, reqno]{amsart}
\usepackage{amsmath,amsthm,amssymb,mathrsfs}
\usepackage[abbrev,non-sorted-cites]{amsrefs}
\usepackage{color,graphicx}
\usepackage{verbatim}
\usepackage{datetime}
\usepackage{hyperref}

\theoremstyle{plain}
  \newtheorem{thm}{Theorem}[section]
  \newtheorem{lem}[thm]{Lemma}
  \newtheorem{prop}[thm]{Proposition}
  \newtheorem{cor}[thm]{Corollary} 
  \newtheorem*{thma}{Theorem A}
  \newtheorem*{thmb}{Theorem B}
\theoremstyle{definition}
  \newtheorem{defn}[thm]{Definition}
  
  \newtheorem{rmk}[thm]{Remark}
  \newtheorem{claim}[thm]{Claim}
  \newtheorem*{ack*}{Acknowledgement}
  \newtheorem*{ques*}{Question}
\theoremstyle{plain}

\numberwithin{equation}{section}

\newcommand\ip[2]{\langle{#1},{#2}\rangle}

\newcommand\pl{\partial}

\newcommand\oh{\frac{1}{2}}
\newcommand\dd{{\mathrm d}}

\newcommand\w{\wedge}
\newcommand\sm{\sigma}
\newcommand\dt{\delta}

\newcommand\vep{\varepsilon}
\newcommand\vph{\varphi}
\newcommand\om{\omega}
\newcommand\ta{\theta}
\newcommand\gm{\gamma}

\newcommand\af{\alpha}
\newcommand\bt{\beta}

\newcommand\ld{\lambda}
\newcommand\vsm{\varsigma}

\newcommand\bld{\boldsymbol{\lambda}}

\newcommand\Om{\Omega}
\newcommand\Sm{\Sigma}

\newcommand\Ld{\Lambda}

\newcommand\Dt{\Delta}

\newcommand\CP{\mathcal{P}}

\newcommand\CS{\mathcal{S}}

\newcommand\CG{\mathcal{G}}

\newcommand\RSO{\mathrm{SO}}

\newcommand\RSU{\mathrm{SU}}
\newcommand\RGL{\mathrm{GL}}

\newcommand\BC{\mathbb{C}}

\newcommand\BR{\mathbb{R}}
\newcommand\BN{\mathbb{N}}

\newcommand\BH{\mathbb{H}}

\newcommand\fs{\mathfrak{s}}

\newcommand\fa{\mathfrak{a}}
\newcommand\fb{\mathfrak{b}}
\newcommand\fc{\mathfrak{c}}
\newcommand\fd{\mathfrak{d}}
\newcommand\fe{\mathfrak{e}}
\newcommand\fl{\mathfrak{l}}

\newcommand\fp{\mathfrak{p}}
\newcommand\fq{\mathfrak{q}}
\newcommand\fr{\mathfrak{r}}
\newcommand\ff{\mathfrak{f}}
\newcommand\fg{\mathfrak{g}}

\newcommand\td{\tilde}

\newcommand\op{\oplus}

\newcommand\bx{\mathbf{x}}
\newcommand\by{\mathbf{y}}

\newcommand\bh{\mathbf{h}}

\newcommand\bL{\mathbf{L}}
\newcommand\bM{\mathbf{M}}
\newcommand\bN{\mathbf{N}}

\newcommand\bfI{\mathbf{I}}

\DeclareMathOperator{\tr}{tr}

\DeclareMathOperator{\re}{Re}
\DeclareMathOperator{\im}{Im}

\DeclareMathOperator{\spin}{Spin}

\begin{document}
\title{Bernstein Theorems for Calibrated Submanifolds\\in $\BR^7$ and $\BR^8$}
\subjclass{Primary 53C38, Secondary 53A10, 49Q05, 35J60}
\author{Chun-Kai Lien}
\address{Department of Mathematics, National Taiwan University, Taipei 10617, Taiwan}
\email{d13221002@ntu.edu.tw }

\author{Chung-Jun Tsai}
\address{Department of Mathematics, National Taiwan University, and National Center for Theoretical Sciences, Math Division, Taipei 10617, Taiwan}
\email{cjtsai@ntu.edu.tw}
\begin{abstract}
    This paper explores the Bernstein problem of smooth maps $f:\BR^4\to\BR^3$ whose graphs form coassociative submanifolds in $\BR^7$. We establish a condition, expressed in terms of the second elementary symmetric polynomial of the map's slope, that ensures $f$ is affine.  A corresponding result is also established for Cayley submanifolds in $\BR^8$.
\end{abstract}

\maketitle

\section{Introduction}

The classical Bernstein theorem asserts that a complete minimal surface which is the graph of $f:\BR^2\to\BR^1$ must be a plane.  There have been many studies on this Bernstein type theorem in the general setting: suppose that a complete minimal submanifold in $\BR^{n+m}$ is the graph of some $f:\BR^n\to\BR^m$, is $f$ an affine function?  It is better understood in the hypersurface case, $m=1$.  One may consult \cite{EH90} and the references therein for the hypersurface case.

In higher codimensions, Lawson and Osserman \cite{LO77}*{Section 7} constructed a minimal cone in $\BR^7$, which is the graph of a Lipschitz map from $\BR^4$ to $\BR^3$.  The map is explicitly defined by
\begin{align} \label{LOmap}
    f(\bx) &= \frac{\sqrt{5}}{2}\frac{1}{|\bx|} (x_0^2+x_1^2-x_2^2-x_3^2,2x_1x_2+2x_0x_3,2x_1x_3-2x_0x_2)
\end{align}
for $\bx = (x_0,x_1,x_2,x_3)\in\BR^4$.  Consequently, the Bernstein type theorem cannot hold in its most general form in higher codimensions.  Known results often involve the (non-negative) square root of the eigenvalues of $(\dd f)^T(\dd f)$, namely, the singular values of $\dd f$.  Denote the singular values by $\{\ld_1,\cdots,\ld_{\min\{n,m\}}\}$.  For \eqref{LOmap}, the singular values are $\{\sqrt{5},\sqrt{5},\frac{\sqrt{5}}{2}\}$.  Notable previous results include:
\begin{itemize}
    \item For $n = 3$ and general $m$, Fischer-Colbrie \cite{FC80}*{Theorem 5.4} proved the Bernstein theorem under the condition of bounded gradient ($\ld_i\leq C$ for all $i$).
    \item For general $n, m$, Wang \cite{Wang03}*{Theorem 1.1} established the Bernstein theorem for maps with bounded gradient and strict area-decreasing property, specifically, $\ld_i\ld_j \leq 1-\vep$ for some $\vep>0$ and any $i\neq j$.
    \item For Lagrangian minimal submanifold in $\BR^{n}\op\BR^{n}$, Tsui and Wang \cite{TW02}*{Theorem A} proved the Bernstein theorem when the gradient is bounded and $\ld_i\ld_j \geq -1$ for any $i\neq j$.  In the Lagrangian case, $f = \nabla u$ for some $u:\BR^n\to \BR^1$, and the singular values are defined as the eigenvalues of $D^2u$.  In particular, $\ld_i$ could be negative.  A similar result was also obtained by Yuan \cite{Yuan02}.
\end{itemize}

Minimal Lagrangian submanifolds, also known as special Lagrangian submanifolds, are volume minimizers and a key class of calibrated submanifolds introduced by Harvey and Lawson \cite{HL82}.  In \cite{HL82}, the volume-minimizing property of special Lagrangians is proved by leveraging the complex structure of $\BR^n\op\BR^n\cong\BC^n$.

Harvey and Lawson \cite{HL82} identified three additional types of volume minimizers:
\begin{center}
\begin{tabular}{c c l} 
& dimension & ambient space \\
\hline
coassociative & $4$ & $\BR^7 \cong \BH\op\im\BH$ \\
associative & $3$ & $\BR^7 \cong \im\BH\op\BH$ \\
Cayley & $4$ & $\BR^8 \cong \BH\op\BH$
\end{tabular}
\end{center}
They are defined using the octonionic structure of their respective ambient spaces.  Their definitions will be briefly reviewed in Section \ref{sec_LA7}, \ref{sec_LA8} and \ref{sec_sfform}.  Note that associative submanifolds are of $3$ dimensions.  Therefore, if an associative submanifold is the graph of an entire map from $\BR^3$ to $\BR^4$ with bounded gradient, Fischer-Colbrie's result \cite{FC80}*{Theorem 5.4} can be applied to conclude its flatness.  Non-flat examples with unbounded gradients are not hard to be constructed, such as by modifying the example in \cite{Yuan02}*{p.118}.

\subsection{The Main Results}

This paper examines the Bernstein problem for coassociative submanifolds in $\BR^7$ and Cayley submanifolds in $\BR^8$.  Given the Lawson--Osserman cone, defined by the graph of \eqref{LOmap}, is coassociative \cite{HL82}*{IV.3}, it is clear that the bounded gradient alone is insufficient in the coassociative case.

For coassociative submanifolds graphical over $\BH\op\{0\}$, their singular values can be shown to satisfy
\begin{align}
    \ld_1 + \ld_2 + \ld_3 &= \ld_1\ld_2\ld_3 ~.
\label{coass_SV} \end{align}
Similar to special Lagrangian submanifolds, $\ld_i$ could be negative.  However, unlike the special Lagrangian case, sign flipping altogether is an ambiguity (see Remark \ref{rmk_coass_SVD}).  It suggests that conditions based on quadratic expressions of $\ld_i$ are more useful.

Moduli sign flipping, the locus of \eqref{coass_SV} in $\BR^3$ comprises two components (Lemma \ref{lem_SVD_3d}): one containing the origin, and the one in the first octant.  Here is our result in the coassociative case.

\begin{thma}[Theorem \ref{thm_Bernstein_coass}]
    Suppose that a complete, coassociative submanifold in $\BH\op\im\BH$ is the graph of a map from $\BH$ to $\im\BH$, whose singular values satisfy one of the following conditions:
    \begin{itemize}
        \item The singular values belong to the component containing the origin, and $\ld_1\ld_2+\ld_2\ld_3+\ld_3\ld_1\geq -2\sqrt{2} + \vep$ for some $0<\vep<\!<1$.
        \item Or, the singular values are all positive (or all negative), and $\ld_i\ld_j \leq 4$ for any $i\neq j$.
    \end{itemize}
    Then, the coassociative submanifold is an affine space.
\end{thma}

For a Cayley submanifold which is graphical over $\BH\op\{0\}$, its singular values must obey
\begin{align}
    \ld_0 + \ld_1 + \ld_2 + \ld_3 &= \ld_1\ld_2\ld_3 + \ld_2\ld_3\ld_0 + \ld_3\ld_1\ld_0 + \ld_0\ld_1\ld_2 ~.
\label{Cayley_SV} \end{align}
In this case, we have the following theorem.

\begin{thmb}[Theorem \ref{thm_Bernstein_Cayley}]
    Suppose that a complete, Cayley submanifold in $\BH\op\BH$ is the graph of a map from $\BH$ to $\BH$, whose singular values satisfy
    \begin{align*}
        -\sqrt{6}+\vep \leq \ld_0\ld_1 + \ld_0\ld_2 + \ld_0\ld_3 + \ld_1\ld_2 + \ld_1\ld_3 + \ld_2\ld_3 \leq 0
    \end{align*}
    for some $0<\vep<\!<1$.  Then, the Cayley submanifold is an affine space.
\end{thmb}

We briefly compare our conditions with the area-decreasing condition, $|\ld_i\ld_j| \leq 1$.  Initially, the conditions of Theorem A and B might appear more restrictive than the area-decreasing condition.  However, this is not the case.  The area-decreasing condition has to be considered in conjunction with \eqref{coass_SV} or \eqref{Cayley_SV}.  For example, $(\ld_1,\ld_2,\ld_3) = (\sqrt[4]{8}, -\sqrt[4]{8}, 0)$ does not satisfy the area-decreasing.
By expressing $\ld_3$ as a function of $\ld_1$ and $\ld_2$ using \eqref{coass_SV}, one can visualize the conditions on the $\ld_1\ld_2$-plane.  In Figure \ref{areadec}, the solid curve enclosed the region corresponding to the first condition in Theorem A (with $\vep = 0$), while the dotted curves enclose the hexagon-like region representing the area-decreasing condition.  It is also not hard to see that our condition is better than the condition in \cite{JXY18}.

\begin{figure}
\centering
    \includegraphics[width=0.35\textwidth]{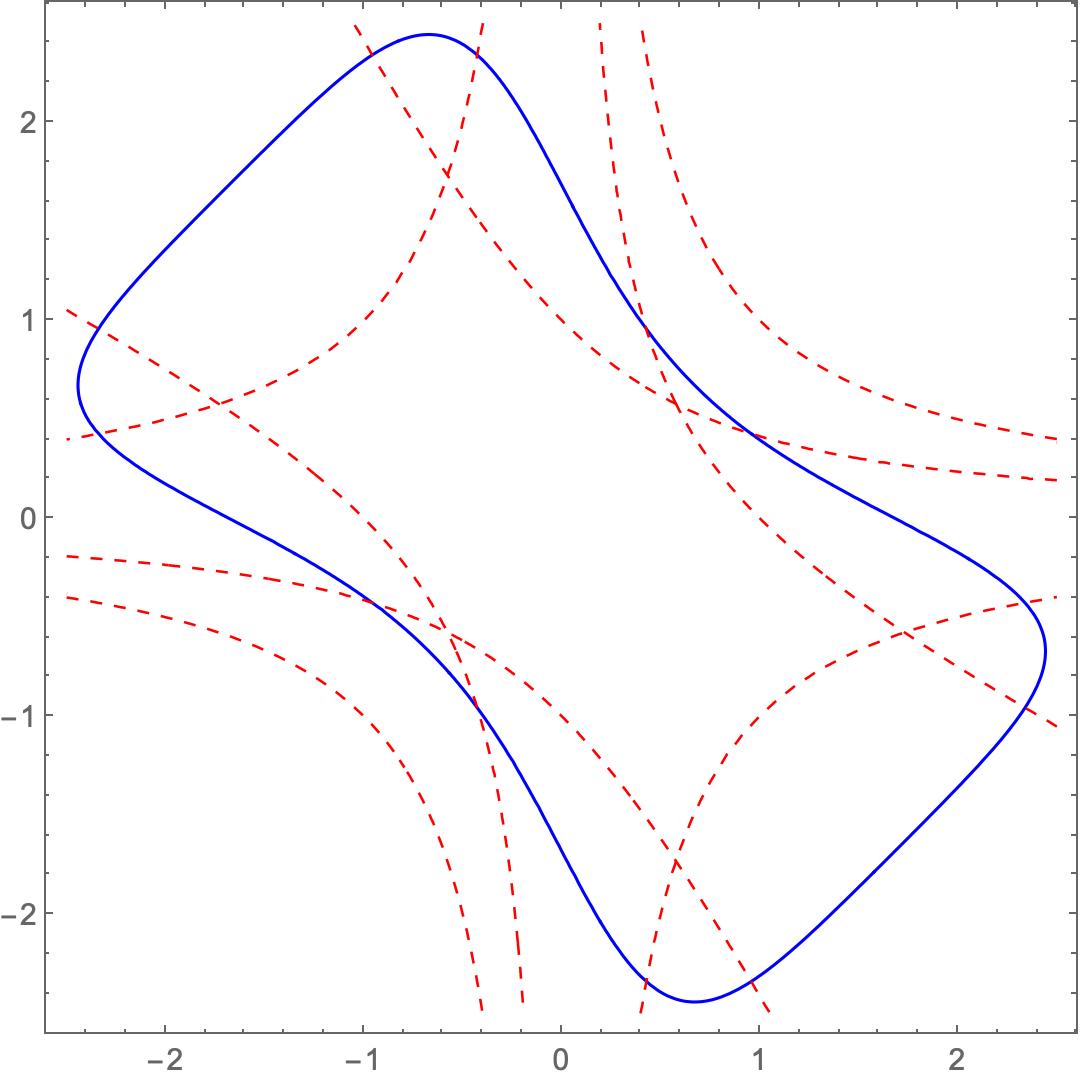}
    \caption{The conditions}
    \label{areadec}
\end{figure}

This paper is organized as follows:  Section \ref{sec_SVD} examines the properties of the loci defined by \eqref{coass_SV} and \eqref{Cayley_SV}.  Section \ref{sec_LA7}, \ref{sec_LA8} and \ref{sec_sfform} review the preliminaries of calibrated submanifolds in $\BR^7$ and $\BR^8$, and establish the singular value decompositions.  For potential future use, the discussion of associative submanifolds is also included.  Section \ref{sec_Bernstein_coass} provides the proof of Theorem A.  Section \ref{sec_Bernstein_Cayley} presents the proof of Theorem B, complemented by a technical lemma in Appendix \ref{sec_claim}.

\begin{ack*}
    The authors are grateful to Mao-Pei Tsui and Mu-Tao Wang for helpful discussions.  This research is supported in part by the Taiwan NSTC grant 112-2628-M-002-004-MY4.
\end{ack*}

\section{Basic Properties of the Set of Singular Values} \label{sec_SVD}

For $j\in\{0,\ldots,n\}$, denote by $\sm_j$ the $j^{\text{th}}$ elementary symmetric polynomial in $n$-variables.  For example, $\sm_0(\ld_1,\ldots,\ld_n) = 1$, $\sm_1(\ld_1,\ldots,\ld_n) = \sum_{i=1}^n\ld_j$, $\sm_2(\ld_1,\ldots,\ld_n) = \sum_{1\leq i<j\leq n}\ld_i\ld_j$, and $\sm_n(\ld_1,\ldots,\ld_n) = \prod_{i=1}^n\ld_i$.

\begin{lem} \label{lem_SVD_3d}
    The set $\CG = \{\bld = (\ld_1,\ld_2,\ld_3)\in\BR^3 : \sm_1(\bld) = \sm_3(\bld) \}$ is a smooth hypersurface in $\BR^3$.  It has the following properties.
    \begin{enumerate}
        \item $\CG$ has three connected components, $\CG_+\cup\CG_0\cup\CG_-$.  The component $\CG_+$ is the intersection of $\CG$ with the first octant, $\CG = \{(\ld_1,\ld_2,\ld_3)\in\CG : \ld_i>0 \text{ for all }i\}$, and $\CG_- = -\CG_+$.
        \item For any $i\neq j$, $\ld_i\ld_j < 1$ on $\CG_0$, and $\ld_i\ld_j > 1$ on $\CG_+\cup\CG_-$.
        \item $\sm_2(\bld) \leq 0$ on $\CG_0$; $\sm_2(\bld) \geq 9$ on $\CG_+\cup\CG_-$.
    \end{enumerate}
\end{lem}
\begin{proof}
    Rewrite the equation as $\ld_1 + \ld_2 = (\ld_1\ld_2 - 1)\ld_3$.  If $\ld_1\ld_2 = 1$, $\ld_1+\ld_2 = 0$.  It follows that $\ld_1\ld_2\leq0$, which contradicts to $\ld_1\ld_2 = 1$.  Hence, $\ld_1\ld_2 \neq 1$.

    As a consequence,
    \begin{align*}
        \CG &= \{ (\ld_1,\ld_2,\frac{\ld_1+\ld_2}{\ld_1\ld_2-1})\in\BR^3 : \ld_1\ld_2 \neq 1 \} ~.
    \end{align*}
    It is not hard to see that when $\ld_1\ld_2<1$, $\ld_1,\ld_2,\ld_3$ cannot be of the same sign.  This finishes the proof of assertion (i) and (ii).  Note that on $\CG_0$, $\ld_i\ld_j<1$ for any $i\neq j$.

    For the first part of (iii),
    \begin{align*}
        \ld_1\ld_2 + (\ld_1+\ld_2)\frac{\ld_1+\ld_2}{\ld_1\ld_2-1} = \frac{1}{\ld_1\ld_2 - 1} \left( \ld_1^2\ld_2^2 + \oh\ld_1^2 + \oh\ld_2^2 + \oh(\ld_1+\ld_2)^2 \right) ~.
    \end{align*}
    Since $\ld_1\ld_2 < 1$ on $\CG_0$, $\ld_1\ld_2 + \ld_2\ld_3 + \ld_3\ld_1 \leq 0$ on $\CG_0$.
    For the second part of (iii), it suffices to do it on $\CG_+$.  By completing the square and the AM-GM inequality,
    \begin{align*}
        (\ld_1 + \ld_2 +\ld_3)^2 &\geq 3(\ld_1\ld_2 + \ld_2\ld_3 + \ld_3\ld_1)\\
        &\geq 9(\ld_1\ld_2\ld_3)^{\frac{2}{3}} = 9(\ld_1 + \ld_2 + \ld_3)^{\frac{2}{3}} ~.
    \end{align*}
    It follows that $\ld_1 + \ld_2 + \ld_3 \geq 3\sqrt{3}$, and the assertion follows.  Note that the equality is achieved at $\ld_1 = \ld_2 = \ld_3 = \sqrt{3}$.
\end{proof}

We will need the following properties later, which does not require $\sm_1(\bld) = \sm_3(\bld)$, and whose proof is simply by completing the square.
\begin{lem} \label{lem_sms_3D}
    For $\bld = (\ld_1,\ld_2,\ld_3)\in\BR^3$, \textnormal{(i)}~ $(\sm_1(\bld))^2 \geq 3\sm_2(\bld)$, and \textnormal{(ii)}~ $(\sm_2(\bld))^2 \geq 3\sm_1(\bld)\sm_3(\bld)$.
\end{lem}


The following lemma is the four-dimensional version of Lemma \ref{lem_SVD_3d}.
\begin{lem} \label{lem_SVD_4d}
    The set $\CS = \{\bld = (\ld_0,\ld_1,\ld_2,\ld_3)\in\BR^4 : \sm_1(\bld) = \sm_3(\bld) \}$ is a smooth hypersurface in $\BR^4$.  It has the following properties.
    \begin{enumerate}
        \item $\sm_2(\bld|i) \neq 1$ for all $i$. The notation means $\sigma_2$ on $\bld$ with $\ld_i$ removed.  For example, $\sm_2(\bld|1) = \ld_0\ld_2 + \ld_2\ld_3 + \ld_3\ld_0$.
        \item $\CS$ has three connected components, $\CS_+\cup\CS_0\cup\CS_-$, where $\CS_0 = \{\bld\in\CS : \sm_2(\bld|i) < 1 \text{ for all } i \}$, $\CS_+ = \{\bld\in\CS : \sm_2(\bld|i) > 1 \text{ for all } i \text{ and } \sm_1(\bld)\geq 4\}$ and $\CS_- = -\CS_+$.
        \item $\sm_2 \leq 0$ on $\CS_0$; $\sm_2 \geq 6$ on $\CS_+\cup\CS_-$.
    \end{enumerate}
\end{lem}
\begin{proof}
    The equation can be rewritten as $\sm_1(\bld|0) - \sm_3(\bld|0) = (\sm_2(\bld|0) - 1)\ld_0$.  Suppose that $\sm_2(\bld|0) = 1$.  One must have $\sm_1(\bld|0) - \sm_3(\bld|0) = 0$.  By Lemma \ref{lem_SVD_3d} (iii), $\sm_2(\bld|0)$ can never be $1$, which is a contradiction.

    To avoid confusion, denote $\sm_i(\bld|0)$ by $\td{\sm}_i(\ld_1,\ld_2,\ld_3)$.  It follows that $\CS$ is graphical over $\{0\}\times\BR^3$:
    \begin{align*}
        \CS &= \{ (\frac{(\td{\sm}_1-\td{\sm}_3)(\ld_1,\ld_2,\ld_3)}{(\td{\sm}_2-1)(\ld_1,\ld_2,\ld_3)},\ld_1,\ld_2,\ld_3)\in\BR^4 : \sm_2(\ld_1,\ld_2,\ld_3) \neq 1 \} ~.
    \end{align*}
    It is not hard to see that the domain, $\{(\ld_1,\ld_2,\ld_3)\in\BR^3:\td{\sm}_2(\ld_1,\ld_2,\ld_3) \neq 1\}$ has three connected components: $\pi(\CS_+)\ni(1,1,1)$, $\pi(\CS_0)\ni(0,0,0)$ and $\pi(\CS_-)=-\pi(\CS_+)\ni(-1,-1,-1)$, where $\pi$ is the orthogonal projection onto $\{0\}\times\BR^3$.

    Since $\sm_2(\bld|i) - 1$ is nonzero on $\CS$, it cannot change sign on $\CS_0$.  Note that it is negative at $\bld=(0,0,0,0)$, and thus $\sm_2(\bld|i) < 1$ on $\CS_0$ for $i=1,2,3$.  The same argument implies that $\sm_2(\bld|i) > 1$ on $\CS_+\cup\CS_-$ for $i=1,2,3$.
    
    For $\bld\in\CS_+$, it follows from Lemma \ref{lem_sms_3D} (i) that $\td{\sm}_1\geq\sqrt{3\td{\sm}_2} > \sqrt{3}$.  According to Lemma \ref{lem_sms_3D} (ii), $\td{\sm}_3 \leq {(\td{\sm}_2)^2}/{(3\td{\sm}_1)} \leq (\td{\sm}_2/\sqrt{3})^{\frac{3}{2}}$.  Hence, for $\bld\in\CS_+$,
    \begin{align*}
        \sm_1 - 4 &= \frac{\td{\sm}_1-\td{\sm}_3}{\td{\sm}_2-1} + \td{\sm}_1 - 4 \\
        &= \frac{1}{\td{\sm}_2-1} \left[ -\td{\sm}_3 + \td{\sm}_1\td{\sm}_2 - 4\td{\sm}_2 + 4 \right] \\
        &\geq \frac{1}{\td{\sm}_2-1} \left[ -\frac{1}{3\sqrt{3}}(\td{\sm}_2)^{\frac{3}{2}} + \sqrt{3}(\td{\sm}_2)^{\frac{3}{2}} - 4\td{\sm}_2 + 4 \right] \\
        &= \frac{1}{\td{\sm}_2-1} \left( \frac{8}{3\sqrt{3}}\sqrt{\td{\sm}_2} + \frac{4}{3} \right) \left( \sqrt{\td{\sm}_2} - \sqrt{3} \right)^2 \geq 0 ~. 
    \end{align*}

    It remains to prove assertion (iii).  When $\bld\in\CS_0$, $\td{\sm}_2<1$, and
    \begin{align*}
        \sm_2 &= \frac{\td{\sm}_1-\td{\sm}_3}{\td{\sm}_2-1}\td{\sm}_1+ \td{\sm}_2 \\
        &= \frac{1}{\td{\sm}_2-1} \left[ ((\td{\sm}_1)^2-\td{\sm}_2) + ((\td{\sm}_2)^2-\td{\sm}_1\td{\sm}_3) \right] ~.
    \end{align*}
    Due to Lemma \ref{lem_sms_3D}, the numerator is no less than $\frac{2}{3}[(\td{\sm}_1)^2+(\td{\sm}_2)^2]$, and thus $\sm_2 \leq 0$.  When $\bld\in\CS_+\cup\CS_-$, again by Lemma \ref{lem_sms_3D},
    \begin{align*}
        \sm_2 - 6 &= \frac{1}{\td{\sm}_2-1} \left[ (\td{\sm}_1)^2 - \td{\sm}_1\td{\sm}_3 + (\td{\sm}_2-1)(\td{\sm}_2-6) \right] \\
        &\geq \frac{1}{\td{\sm}_2-1} \left[ 3\td{\sm}_2 - \frac{1}{3}(\td{\sm}_2)^2 + (\td{\sm}_2)^2 - 7\td{\sm}_2 + 6 \right] \\
        &= \frac{1}{\td{\sm}_2-1}\,\frac{2}{3}(\td{\sm}_2-3)^2 \geq 0 ~.
    \end{align*}
    This finishes the proof of the lemma.
\end{proof}

\section{Some Linear Algebra about the $G_2$ Geometry of $\BR^7$} \label{sec_LA7}
We follow \cite{K05} and \cite{SW17} closely in our discussion of $G_2$ and $\spin(7)$ geometry.  Readers can find more detailed information and additional references there.

Let $(x^0,x^1,x^2,x^3,y^1,y^2,y^3)$ be the coordinate for $\BR^7$.  For $i = 1,2,3$, denote by $\om^i$ the $2$-form $\dd x^0\w\dd x^i + \dd x^j\w\dd x^k$ where $(i,j,k)$ is a cyclic permutation of $(1,2,3)$.  Let
\begin{align} \label{g2_form}
    \vph &= \dd y^{123} - \dd y^1\w\om^1 - \dd y^2\w\om^2 - \dd y^{3}\w\om^3 ~,
\end{align}
where $\dd y^{123}$ denotes $\dd y^1\w\dd y^2\w\dd y^3$.  With respect to the standard metric $\sum_{i=0}^3(\dd x^i)^2 + \sum_{j=1}^3(\dd y^j)^2$ and the volume form $\dd x^{0123}\w\dd y^{123}$, the Hodge star of $\vph$ is
\begin{align} \label{g2_star}
    \psi &= \dd x^{0123} - \dd y^{23}\w\om^1 - \dd y^{31}\w\om^2 - \dd y^{12}\w\om^3 ~.
\end{align}
The exceptional Lie group $G_2$ is the automorphism group of $\vph$:
\begin{align*}
    G_2 &= \{ g\in\RGL(7;\BR) : g^*(\vph) = \vph \} ~.
\end{align*}
Since $\vph$ determines the metric and orientation, $G_2$ is contained in $\RSO(7)$.

Recall that the standard action of $\RSO(4)$ on $\BR^4$ naturally induces an action on $\Ld^4_+\BR^4$.
Considering $\BR^7$ as $\Ld^2_+\BR^4$ (where $\BR^4 = \{ y^1 = y^2 = y^3 = 0\}$), one finds $\RSO(4) < G_2$.  Specifically, for $g\in\RSO(4)$, $g^*(\om^i) = \sum_{j=1}^3 A^i_j\,\om^j$ where $[A]\in\RSO(3)$.  It acts on the $y$-coordinates by $g^*(\dd y^i) = \sum_{j=1}^3 A^i_j\,\dd y^j$.  One can see from \eqref{g2_form} that $g$ leaves $\vph$ invariant.  For $\RSO(3) < \RSO(4)$ defined to be the subgroup that fixes the $x^0$ coordinate, the corresponding action on $\BR^7$ is given by the standard action on the $(x^1,x^2,x^3)$-summand and the $(y^1,y^2,y^3)$-summand.  We fix this choice of subgroups $\RSO(3) < \RSO(4) < G_2$ in this paper.

\subsection{Products and $G_2$-basis}

The metric dual of $\vph$ defines a \emph{vector cross product} by
\begin{align*}
    \ip{u\times v}{w} &= \vph(u,v,w)
\end{align*}
for any $u,v,w\in\BR^7$.  For example, $\frac{\pl}{\pl x^0}\times\frac{\pl}{\pl x^1} = -\frac{\pl}{\pl y^1}$.  An ordered, orthonormal basis $\{e_0,\ldots,e_6\}$ is called a \emph{$G_2$-basis} if their vector cross products satisfy the same relation as that of the above coordinate basis.  Equivalently, if $\{\ta^i\}_{i=0}^6$ is the dual basis of $\{e_i\}_{i=0}^6$, then
\begin{align*}
    \vph &= \ta^4\w\ta^5\w\ta^6 - \ta^4\w(\ta^0\w\ta^1 + \ta^2\w\ta^3) \\
    &\quad - \ta^5\w(\ta^0\w\ta^2 + \ta^3\w\ta^1) - \ta^6\w(\ta^0\w\ta^3 + \ta^1\w\ta^2) ~.
\end{align*}

Similarly, the metric dual of $\psi$ defines the \emph{associator} by
\begin{align*}
    \ip{z}{[u,v,w]} &= 2\psi(z,u,v,w)
\end{align*}
for any $x,u,v,w\in\BR^7$.  For instance, $[\frac{\pl}{\pl x^1},\frac{\pl}{\pl x^2},\frac{\pl}{\pl x^3}] = 2\frac{\pl}{\pl x^0}$ and $[\frac{\pl}{\pl y^1},\frac{\pl}{\pl y^2},\frac{\pl}{\pl y^3}] = 0$.

\subsection{Coassociative Subspaces}

An oriented $4$-dimensional vector subspace $P\subset\BR^7$ is said to be \emph{coassociative} if it is \emph{calibrated} by $\psi$.  Namely, $\psi(P) = \pm1$.  According to \cite{HL82}*{Corollary 1.20 on p.117}, $P$ is coassociative if and only if $\vph|_P = 0$.

\begin{lem} \label{lem_coass_SVD}
    Suppose that $P\subset\BR^7$ is coassociative, and is graphical over $\BR^4\times\{0\}$.  Then, there exists an $\RSO(4) < G_2$ change of coordinate, after which $P$ is spanned by
    \begin{align*}
        \frac{\pl}{\pl x^0} \quad\text{and}\quad \frac{\pl}{\pl x^i} + \ld_i\frac{\pl}{\pl y^i} \quad\text{for $i=1,2,3$} ~,
    \end{align*}
    for some $\ld_i\in\BR$ satisfying $\ld_1 + \ld_2 + \ld_3 = \ld_1\ld_2\ld_3$.
\end{lem}
\begin{proof}
    Since $P$ is given by the graph of a linear map from $\BR^4\times\{0\}$ to $\{0\}\times\BR^3$, one can apply the singular value decomposition to find an orthogonal basis for $P$.  Recall that the singular value decomposition produces an $\RSO(4)$ change of coordinate for the domain, $\BR^4\times\{0\}$.  By the $\RSO(4) < G_2$ action on $\BR^7$, $P$ has the orthogonal basis
    \begin{align*}
        \frac{\pl}{\pl x^0} \quad\text{and}\quad \frac{\pl}{\pl x^i} + \sum_{j=1}^3 B_i^j\frac{\pl}{\pl y^j} \quad\text{for $i=1,2,3$} ~.
    \end{align*}
    In particular, $\sum_{k=1}^3 B^k_i B^k_j = 0$ for any $i\neq j$.

    Since $P$ is coassociative, $\vph|_P = 0$.  It implies that $B_i^j = B_j^i$ for any $i\neq j$, and $\tr(B) = \det(B)$.  Since $B$ is symmetric, it can be diagonalized by an $\RSO(3)$ matrix.  By the $\RSO(3) < G_2$ action on $\BR^7$, the orthogonal basis of $P$ becomes
    \begin{align*}
        \frac{\pl}{\pl x^0} \quad\text{and}\quad \frac{\pl}{\pl x^i} +\ld_i\frac{\pl}{\pl y^i} \quad\text{for $i=1,2,3$} ~.
    \end{align*}
    We infer from $\tr(B) = \det(B)$ that $\ld_1 + \ld_2 + \ld_3 = \ld_1\ld_2\ld_3$.
\end{proof}

\begin{rmk} \label{rmk_coass_SVD}
    The correspondence given by Lemma \ref{lem_coass_SVD}, $P\mapsto(\ld_1,\ld_2,\ld_3)$, has ambiguities.  In addition to relabeling $\ld_i$' s, the negative identity matrix $-\bfI_4\in\RSO(4) < G_2$ flips the sign of all $\ld_i$'s.  Therefore, it induces an equivalence between $\CG_+$ and $\CG_-$ described by Lemma \ref{lem_sms_3D}
    
    Since Lemma \ref{lem_coass_SVD} is only for graphical ones, it does not mean that the Grassmannian of coassociative $4$-planes is disconnected.  For example, $P_t$ spanned by $\frac{\pl}{\pl x^0}$, $\cos t\frac{\pl}{\pl x^1} + \sin t\frac{\pl}{\pl y^1}$, $\cos t\frac{\pl}{\pl x^2} + \sin t\frac{\pl}{\pl y^2}$, and $\cos(2t)\frac{\pl}{\pl x^3} - \sin(2t)\frac{\pl}{\pl y^3}$ is coassociative for all $t$.  It connects different components by passing through non-graphical coassociatives.
\end{rmk}

The components $\CG_0$ and $\CG_+\cup\CG_-$ defined in Lemma \ref{lem_sms_3D} correspond to different orientation in the following sense.
\begin{cor} \label{cor_coass_SVD}
    For any $(\ld_1,\ld_2,\ld_3)\in\CG$, let
    \begin{align} \label{dt_sign}
        \dt &= \begin{cases}
            1 &\text{if } (\ld_1,\ld_2,\ld_3)\in\CG_0 ~, \\
            -1 &\text{if } (\ld_1,\ld_2,\ld_3)\in\CG_+\cup\CG_- ~, \\
        \end{cases}
    \end{align}
    and
    \begin{align*}
        e_0 &= \dt\frac{\pl}{\pl x^0} ~,
        & e_i &= \frac{1}{\sqrt{1+\ld_i^2}}\left(\frac{\pl}{\pl x^i} + \ld_i\frac{\pl}{\pl y^i}\right) ~,
        & e_{3+i} &= \frac{\dt}{\sqrt{1+\ld_i^2}}\left(\frac{\pl}{\pl y^i} - \ld_i\frac{\pl}{\pl x^i}\right)
    \end{align*}
    for $i=1,2,3$.  Then, $\psi(e_0,e_1,e_2,e_3) = 1$, and $\{e_i\}_{i=0}^6$ is a $G_2$-basis.  To phrase it formally, the graphical orientation coincides with the orientation determined by $\psi$ if and only if $1-\ld_i\ld_j>0$.
\end{cor}
\begin{proof}
    By a direct computation,
    \begin{align*}
        \psi\left(\frac{\pl}{\pl x^0}, \frac{\frac{\pl}{\pl x^1} +\ld_1\frac{\pl}{\pl y^1}}{\sqrt{1+\ld_1^2}}, \frac{\frac{\pl}{\pl x^2} +\ld_2\frac{\pl}{\pl y^2}}{\sqrt{1+\ld_2^2}}, \frac{\frac{\pl}{\pl x^3} +\ld_3\frac{\pl}{\pl y^3}}{\sqrt{1+\ld_3^2}}\right) &= \frac{1 - \ld_1\ld_2 - \ld_2\ld_3 - \ld_3\ld_1}{\sqrt{(1+\ld_1^2)(1+\ld_2^2)(1+\ld_3^3)}} ~.
    \end{align*}
    Since $\ld_1+\ld_2+\ld_3 = \ld_1\ld_2\ld_3$, the square of the right hand side is equal to $1$.  By plugging into $(\ld_1,\ld_2,\ld_3) = (0,0,0)$ and $(\sqrt{3},\sqrt{3},\sqrt{3})$, the right hand side is $\dt$.

    The verification of the $G_2$-basis is a straightforward computation.  We do two of them, and the rest are left as an exercise.
    \begin{align*}
        e_4\times e_1 &= \frac{\dt}{1+\ld_1^2} \left(\frac{\pl}{\pl y^1} - \ld_1\frac{\pl}{\pl x^1}\right) \times \left(\frac{\pl}{\pl x^1} + \ld_1\frac{\pl}{\pl y^1}\right) = \dt\frac{\pl}{\pl x^0} = e_0 ~, \\
        e_2\times e_3 &= \frac{1}{\sqrt{1+\ld_2^2}}\left(\frac{\pl}{\pl x^2} + \ld_2\frac{\pl}{\pl y^2}\right) \times \frac{1}{\sqrt{1+\ld_3^2}}\left(\frac{\pl}{\pl x^3} + \ld_3\frac{\pl}{\pl y^3}\right) \\
        &= \frac{1}{\sqrt{(1+\ld_2^2)(1+\ld_3^2)}}\left( (-1+\ld_2\ld_3) \frac{\pl}{\pl y^1} - (\ld_2+\ld_3)\frac{\pl}{\pl x^1}\right) \\
        &= \frac{-1+\ld_2\ld_3}{\sqrt{(1+\ld_2^2)(1+\ld_3^2)}}\left(\frac{\pl}{\pl y^1} - \ld_1\frac{\pl}{\pl x^1}\right) ~.
    \end{align*}
    By using the above computation on $\psi(e_0,e_1,e_2,e_3)$, one can find out that it is exactly $-e_4$.
\end{proof}

\subsection{Associative Subspaces}

An oriented $3$-dimensional vector subspace $Q\subset\BR^7$ is said to be \emph{associative} if it is \emph{calibrated} by $\vph$.  Namely, $\vph(Q) = \pm1$.  According to \cite{HL82}*{Corollary 1.7 on p.114}, $Q$ is associative if and only if $[u,v,w] = 0$ for a basis $u,v,w$ for $Q$.  That is to say, $\vph(u,v,w,\,\cdot\,) = 0$ as a linear functional on $\BR^7$.

\begin{lem} \label{lem_ass_SVD}
    Suppose that $Q\subset\BR^7$ is associative, and is graphical over $\{0\}\times\BR^3$.  Then, there exists an $\RSO(4) < G_2$ change of coordinate, after which $P$ is spanned by
    \begin{align*}
        \frac{\pl}{\pl y^i} + \ld_i\frac{\pl}{\pl x^i} \quad\text{for $i=1,2,3$} ~,
    \end{align*}
    for some $\ld_i\in\BR$ satisfying $\ld_1 + \ld_2 + \ld_3 = \ld_1\ld_2\ld_3$.
\end{lem}
\begin{proof}
    Since $\dim Q = 3$, there exists a unit vector in $\BR^4\times\{0\}$ that is perpendicular to $Q$.  By $\RSO(4) < G_2$, we may assume that it is $\frac{\pl}{\pl x^0}$.  Then, $Q$ is the graph of a linear map from the $(y^1,y^2,y^3)$-summand to the $(x^1,x^2,x^3)$-summand.  After applying the singular value decomposition and the $\RSO(3) < G_2$ change of coordinate, $Q$ has the orthogonal basis
    \begin{align*}
        \frac{\pl}{\pl y^i} + \sum_{j=1}^3 B_i^j\frac{\pl}{\pl x^j} \quad\text{for $i=1,2,3$} ~.
    \end{align*}

    A straightforward computation yields
    \begin{align*}
        &\quad \psi\left(\frac{\pl}{\pl y^1} + \sum_{j=1}^3 B_1^j\frac{\pl}{\pl x^j}, \frac{\pl}{\pl y^2} + \sum_{j=1}^3 B_2^j\frac{\pl}{\pl x^j}, \frac{\pl}{\pl y^3} + \sum_{j=1}^3 B_3^j\frac{\pl}{\pl x^j}, \,\cdot\, \right) \\
        &= (-\det(B) + \tr(B))\dd x^0 + (B_3^2 - B_2^3)\dd x^1 + (B_1^3 - B_3^1)\dd x^2 + (B_2^1 - B_1^2)\dd x^3 \\
        &\quad + (B_1^1B_2^3 - B_1^3B_2^1 - B_1^1B_3^2 + B_1^2B_3^1)\dd y^1 + (B_1^2B_2^3 - B_1^3B_2^2 - B_2^1B_3^2 + B_2^2B_3^1)\dd y^2 \\
        &\qquad + (B_1^2B_3^3 - B_1^3B_3^2 + B_2^3B_3^1 - B_2^1B_3^3)\dd y^3 ~.
    \end{align*}
    Since $Q$ is associative, the above linear functional vanishes, which means $\tr(B) = \det(B)$ and $B$ is symmetric.  By the same argument as that for Lemma \ref{lem_coass_SVD}, this lemma follows.
\end{proof}

\section{Some Linear Algebra about the $\spin(7)$ Geometry of $\BR^8$} \label{sec_LA8}

Let $(x^0,\ldots,x^3,y^0,\ldots,y^3)$ be the coordinate for $\BR^8$.  The \emph{Cayley form} is
\begin{align} \label{spin_form}
    \Phi &= \dd x^{0123} + \dd y^{0123} - \om^1\w\gm^1 - \om^2\w\gm^2 - \om^3\w\gm^3
\end{align}
where
\begin{align*}
    \om^i = \dd x^0\w\dd x^i + \dd x^j\w\dd x^k \quad\text{ and }\quad
    \gm^i = \dd y^0\w\dd y^i + \dd y^j\w\dd y^k
\end{align*}
for $i = 1,2,3$ and $(i,j,k)$ a cyclic permutation of $(1,2,3)$.

The automorphism group of $\Phi$ is the Lie group $\spin(7)$: $\spin(7) = \{ g\in\RGL(8;\BR) : g^*(\Phi) = \Phi \}$.  The following subgroups of $\spin(7)$ will be used later.
\begin{itemize}
    \item For $h \in \RSU(2)$, $\rho_x^-(h)$ acts on $\{0\}\times\BR^4$ trivially, and acts on $\BR^4\times\{0\}$ by
    \begin{align*}
        \begin{bmatrix} x^0+\sqrt{-1}x^1 & -x^2+\sqrt{-1}x^3 \\ x^2+\sqrt{-1}x^3 & x^0-\sqrt{-1}x^1 \end{bmatrix} &\mapsto h\cdot \begin{bmatrix} x^0+\sqrt{-1}x^1 & -x^2+\sqrt{-1}x^3 \\ x^2+\sqrt{-1}x^3 & x^0-\sqrt{-1}x^1 \end{bmatrix} ~.
    \end{align*}
    The action leaves $\om^i$'s invariant, and thus $\rho_x^-(\RSU(2)) < \spin(7)$.  Note that it acts transitively on the unit vectors of $\BR^4\times\{0\}$.
    \item Analogously to $\rho_x^-$, $\rho_y^-(\RSU(2)) < \spin(7)$, and acts transitively on the unit vectors of $\{0\}\times\BR^4$.
    \item For $h \in \RSO(4)$, let $\td{\rho}(\bx,\by) = (h\bx,h\by)$.  In terms of matrix representative, $\td{\rho}(h)$ is the doubling of $h$ as a $4\times4$-block diagonal matrix.
\end{itemize}

\subsection{Cayley Subspaces}

The metric dual of $\Phi$ defines a triple cross product by
\begin{align*}
    \ip{z}{u\times v\times w} &= \Phi(z,u,v,w) ~.
\end{align*}
An ordered, orthonormal basis $\{e_0,\ldots,e_7\}$ is called a \emph{$\spin(7)$-basis} if their vector cross products satisfy the same relation as that of the basis given by the above coordinate.

An oriented $4$-dimensional vector subspace $R\subset\BR^7$ is said to be \emph{Cayley} if it is \emph{calibrated} by $\Phi$, $\Phi(R) = \pm1$.  By \cite{HL82}*{Proposition 1.25 on p.118}, a $4$-plane $R$ is Cayley if and only if $u\times v\times w\in R$ for any $u,v,w\in R$.

\begin{lem} \label{lem_Cayley_SVD}
    Suppose that $R\subset\BR^8$ is Cayley, and is graphical over $\BR^4\times\{0\}$.  Then, there exists an $\spin(7)$ change of coordinate by using $\rho_x^-(\RSU(2))$, $\rho_y^-(\RSU(2))$ and $\td{\rho}(\RSO(4))$, after which $R$ is spanned by
    \begin{align*}
        \frac{\pl}{\pl x^i} + \ld_i\frac{\pl}{\pl y^i} \quad\text{for $i=0,1,2,3$} ~,
    \end{align*}
    for some $\ld_i\in\BR$ satisfying $\ld_0 + \ld_1 + \ld_2 + \ld_3 = \ld_1\ld_2\ld_3 + \ld_0\ld_2\ld_3 + \ld_0\ld_3\ld_1 + \ld_0\ld_1\ld_2$.
\end{lem}
\begin{proof}
    Since $R$ is graphical over $\BR^4\times\{0\}$, the singular value decomposition finds an ordered, orthogonal basis for $R$, whose projection on $\BR^4\times\{0\}$ is also an oriented, orthonormal basis, and the projection onto $\{0\}\times\BR^4$ is an orthogonal set.  We call them the SVD basis for $R$.
    
    One only needs to consider that case that $R$ is not $\BR^4\times\{0\}$.  In this case, there exists an element , $u_0$,in the SVD basis with non-trivial projection onto $\{0\}\times\BR^4$.  Since $\rho_x^-(\RSU(2))$ and $\rho_y^-(\RSU(2))$ act transitively on $\BR^4\times\{0\}$ and $\{0\}\times\BR^4$, respectively, this element can be assumed to be
    \begin{align*}
        u_0 &= \frac{\pl}{\pl x^0} + \ld_0\frac{\pl}{\pl y^0}
    \end{align*}
    for some $\ld_0\neq 0$.  Let $\RSO(3) < \RSO(4)$ be the stabilizer of the $x^0$-direction.  By an $\td{\rho}(\RSO(3))$-change of coordinate, the rest elements, $\{u_i\}_{i=1}^3$, of the SVD basis can be assumed to be
    \begin{align*}
        u_i &= \frac{\pl}{\pl x^i} + \sum_{j=1}^3 B_i^j\frac{\pl}{\pl y^j}
    \end{align*}
    for $i=1,2,3$.

    Since $R$ is Cayley and $\{u_i\}_{i=0}^3$ is orthogonal, $u_1\times u_2\times u_3$ is parallel to $u_0$.  The vanishing of the $\frac{\pl}{\pl y^i}$-component of $u_1\times u_2\times u_3$ for $i=1,2,3$ implies that $B$ is symmetric.  By using the $\td{\rho}(\RSO(3))$-change of coordinate again, $B$ can be assumed to be diagonal, and $u_i = \frac{\pl}{\pl x^i} + \ld_i\frac{\pl}{\pl y^i}$ for $i=1,2,3$.

    For an orthogonal basis of a Cayley plane, \cite{HL82}*{Theorem 1.24 on p.118} asserts that $\Phi(u_0,u_1,u_2,u_3) = \pm|u_0|\,|u_1|\,|u_2|\,|u_3|$.  It follows that
    \begin{align} \label{eqn_Cayley_basic} \begin{split}
        &\quad 1 + \ld_0\ld_1\ld_2\ld_3 - \ld_0\ld_1 - \ld_2\ld_3 - \ld_0\ld_2 - \ld_3\ld_1 - \ld_0\ld_3 - \ld_1\ld_2 \\
        & = \pm\sqrt{(1+\ld_0^2)(1+\ld_1^2)(1+\ld_2^2)(1+\ld_3^2)} ~.
    \end{split} \end{align}
    Squaring it leads to the relation asserted by this lemma.
\end{proof}

Similar to Corollary \ref{cor_coass_SVD}, the basis given by Lemma \ref{lem_Cayley_SVD} can be extended to a $\spin(7)$-basis.  We will focus on the $\CS_0$ component described in Lemma \ref{lem_SVD_4d}.

\begin{cor} \label{cor_Cayley_SVD}
    For any $(\ld_0,\ld_1,\ld_2,\ld_3)\in\CS_0$, let
    \begin{align*}
        e_i &= \frac{1}{\sqrt{1+\ld_i^2}}\left(\frac{\pl}{\pl x^i} + \ld_i\frac{\pl}{\pl y^i}\right) ~,
        & e_{4+i} &= \frac{1}{\sqrt{1+\ld_i^2}}\left(\frac{\pl}{\pl y^i} - \ld_i\frac{\pl}{\pl x^i}\right)
    \end{align*}
    for $i=0,1,2,3$.  Then, $\{e_i\}_{i=0}^7$ is a $\spin(7)$-basis.  In particular, $\Phi(e_0,e_1,e_2,e_3) = 1$.
\end{cor}
\begin{proof}
    When $(\ld_0,\ld_1,\ld_2,\ld_3)\in\CS_0$, one can argue by continuity that the sign in \eqref{eqn_Cayley_basic} is plus.  With \eqref{eqn_Cayley_basic}, it is a direct computation to verify that $\{e_i\}_{i=0}^7$ is a $\spin(7)$-basis.  For instance,
    \begin{align*}
        \Phi(e_0,e_1,e_2,e_3) &= \frac{1 + \ld_0\ld_1\ld_2\ld_3 - \ld_0\ld_1 - \ld_2\ld_3 - \ld_0\ld_2 - \ld_3\ld_1 - \ld_0\ld_3 - \ld_1\ld_2}{\sqrt{(1+\ld_0^2)(1+\ld_1^2)(1+\ld_2^2)(1+\ld_3^2)}} = 1 ~, \\
        \Phi(e_0,e_5,e_6,e_7) &= \frac{-\ld_1\ld_2\ld_3 + \ld_0 + \ld_1 - \ld_0\ld_2\ld_3 + \ld_2 - \ld_0\ld_1\ld_3 + \ld_3 - \ld_0\ld_1\ld_2}{\sqrt{(1+\ld_0^2)(1+\ld_1^2)(1+\ld_2^2)(1+\ld_3^2)}} = 0 ~,
    \end{align*}
    and the rest are left as an exercise.
\end{proof}

\section{Second Fundamental Form} \label{sec_sfform}

\subsection{Calibrated Submanifolds in $\BR^7$}

A $4$-dimensional submanifold $X\subset\BR^7$ is called a \emph{coassociative submanifold} if $T_pX$ is coassociative for every $p\in X$.  A $3$-dimensional submanifold $Y\subset\BR^7$ is called an \emph{associative submanifold} if $T_pY$ is associative for every $p\in Y$.  Since they are calibrated submanifolds (see \cite{HL82}), they have vanishing mean curvature.  They come from the algebraic structure of octonions, and their second fundamental form admits more symmetry.  

\begin{lem} \label{lem_coass_h}
    Let $X\subset\BR^7$ be a coassociative submanifold.  At a $p\in X$, suppose that $\{e_i\}_{i=0}^3$ is a basis for $T_pX$ and $\{e_\af\}_{\af=4}^6$ is a basis for $N_pX$ such that $\{e_0,\ldots,e_6\}$ is a $G_2$-basis.  Then, the second fundamental form at $p$, $h_{\af ij} = \ip{D_{e_i}e_j}{e_\af}$, satisfies the following relations.
    \begin{align*}
        h_{41i} + h_{52i} + h_{63i} &= 0 ~,  & h_{40i} + h_{53i} - h_{62i} &= 0 ~, \\
        h_{43i} - h_{50i} - h_{61i} &= 0 ~,  & h_{42i} - h_{51i} + h_{60i} &= 0
    \end{align*}
    for $i=0,1,2,3$.
\end{lem}
\begin{proof}
    Extend $e_0, e_1, e_2$ to an orthonormal field for $TX$ over a neighborhood of $p$.  Let $e_3$ be the tangent field of $TX$ such that $\psi(e_0,e_1,e_2,e_3) = 1$.  Let $e_{3+i} = e_i\times e_0$ for $i=1,2,3$.  Since $X$ is coassociative,  this gives a $G_2$-basis for every point in this neighborhood of $p$ in $X$.  Since $\vph$ is parallel, so is the vector cross product, and hence
    \begin{align*}
        h_{41i} &= \ip{D_{e_i}e_1}{e_4} = -\ip{e_1}{D_{e_i}e_4} = -\ip{e_1}{D_{e_i}(e_5\times e_6)} \\
        &= -\ip{e_1}{(D_{e_i}e_5)\times e_6 + e_5\times(D_{e_i}e_6)} \\
        &= \ip{e_1\times e_6}{D_{e_i}e_5} - \ip{e_1\times e_5}{D_{e_i}e_6} \\
        &= \ip{e_2}{D_{e_i}e_5} + \ip{e_3}{D_{e_i}e_6} = - h_{52i} - h_{63i} ~.
    \end{align*}
    The other equalities can be proved by similar calculations.
\end{proof}

The second fundamental form of associative submanifolds admit similar symmetries.  The proof is very similar, and is omitted.
\begin{lem} \label{lem_ass_h}
    Let $Y\subset\BR^7$ be an associative submanifold.  At a $p\in Y$, suppose that $\{e_i\}_{i=4}^6$ is a basis for $T_pY$ and $\{e_\af\}_{\af=0}^3$ is a basis for $N_pY$ such that $\{e_0,\ldots,e_6\}$ is a $G_2$-basis.  Then, the second fundamental form at $p$, $h_{\af ij} = \ip{D_{e_i}e_j}{e_\af}$, satisfies the following relations.
    \begin{align*}
        h_{14i} + h_{25i} + h_{36i} &= 0 ~,  & h_{04i} + h_{35i} - h_{26i} &= 0 ~, \\
        h_{34i} - h_{05i} - h_{16i} &= 0 ~,  & h_{24i} - h_{15i} + h_{06i} &= 0
    \end{align*}
    for $i=4,5,6$.
\end{lem}

\subsection{Calibrated Submanifolds in $\BR^8$}

A $4$-dimensional submanifold $Z\subset\BR^8$ is called a \emph{Cayley submanifold} if $T_pZ$ is Cayley for every $p\in Z$.  Unsurprisingly, its second fundamental form also admits more symmetry.

\begin{lem} \label{lem_Cayley_h}
    Let $Z\subset\BR^8$ be a Cayley submanifold.  At a $p\in Z$, suppose that $\{e_i\}_{i=0}^3$ is a basis for $T_pZ$ and $\{e_\af\}_{\af=4}^7$ is a basis for $N_pZ$ such that $\{e_0,\ldots,e_7\}$ is a $\spin(7)$-basis.  Then, the second fundamental form at $p$, $h_{\af ij} = \ip{D_{e_i}e_j}{e_\af}$, satisfies the following relations.
    \begin{align*}
        h_{40i} + h_{51i} + h_{62i} + h_{73i} &= 0 ~,  & h_{41i} - h_{50i} - h_{63i} + h_{72i} &= 0 ~, \\
        h_{42i} + h_{53i} - h_{60i} - h_{71i} &= 0 ~,  & h_{43i} - h_{52i} + h_{61i} - h_{70i} &= 0
    \end{align*}
    for $i=0,1,2,3$.
\end{lem}
\begin{proof}
    For $i,j\in\{0,1,2,3\}$, we compute
    \begin{align*}
        h_{4ji} &= \ip{D_{e_i}e_j}{e_4} = \ip{D_{e_i}e_j}{e_5\times e_6\times e_7} \\
        &= - \ip{e_j}{(D_{e_i}e_5)\times e_6\times e_7} - \ip{e_j}{e_5\times(D_{e_i}e_6)\times e_7} - \ip{e_j}{e_5\times e_6\times(D_{e_i}e_7)} \\
        &= \ip{D_{e_i}e_5}{e_j\times e_6\times e_7} - \ip{D_{e_i}e_6}{e_j\times e_5\times e_7} + \ip{D_{e_i}e_7}{e_j\times e_5\times e_6} \\
        &= - \ip{e_5}{D_{e_i}(e_j\times e_6\times e_7)} + \ip{e_6}{D_{e_i}(e_j\times e_5\times e_7)} - \ip{e_7}{D_{e_i}(e_j\times e_5\times e_6)} ~,
    \end{align*}
    and this lemma follows.
\end{proof}

\section{Bernstein Theorem for Coassociative Graphs} \label{sec_Bernstein_coass}

\subsection{The Equation for a Minimal Graph}

Suppose that $\Sm^n\subset\BR^N$ is an oriented minimal submanifold, and $\Om$ is a parallel $n$-form on $\BR^N$.  It induces a function on $\Sm$, which will be denoted by $*\Om$.  If $\{e_1,\ldots,e_n\}$ is an oriented, orthonormal basis for $T_p\Sm$, $(*\Om)(p) = \Om(e_1,\ldots,e_n)$.  According to \cite{Wang02}*{Proposition 3.1}, it satisfies
\begin{align} \label{Dt_Om} \begin{split}
    \Dt(*\Om) &= -*\Om \cdot (\sum_{\af,i,j}h_{\af ij}^2) + 2\sum_{\af,\bt}\sum_k \left[ \Om(e_\af,e_\bt,e_3,\ldots,e_n) h_{\af 1k}h_{\bt 2k} \right. \\
    &\qquad\qquad \left. + \Om(e_\af,e_2,e_\bt,\ldots,e_n) h_{\af 1k}h_{\bt 3k} + \Om(e_1,\ldots,e_{n-2},e_\af,e_\bt) h_{\af(n-1)k}h_{\bt nk} \right]
\end{split} \end{align}
where $\Dt$ is the Laplacian of $\Sm$, $\{e_\af\}$ is a orthonormal frame for $N\Sm$, and $h_{\af ij} = \ip{D_{e_i}e_j}{e_\af}$ is the second fundamental form.

By \cite{Wang02}*{(3.1)}, the gradient of $*\Om$ is
\begin{align} \label{grad_Om}
    e_k(*\Om) &= \Om(e_\af,e_2\ldots,e_n) h_{\af 1k} + \cdots + \Om(e_1,\ldots,e_{n-1},e_\af) h_{\af nk} ~.
\end{align}

In the rest of this paper, $\Sm^n$ will always be graphical over some $\BR^n\subset\BR^N$, and $\om$ will be taken to be the parallel transport of the volume form of this $\BR^n$.

\subsection{The Equation for Coassociative Graphs}

If $X\subset\BR^7$ is coassociative and is graphical over $\BR^4\times\{0\}$, one can apply Lemma \ref{lem_coass_SVD} with $P = T_pX$ for every $p\in X$.  Together with Corollary \ref{cor_coass_SVD}, it associates a $G_2$-basis and singular values for each point $p\in X$, and we can evaluate the right hand side of \eqref{Dt_Om} and \eqref{grad_Om} at $p$ by using this basis.  Since $X$ is connected, the singular values must be either in $\CG_0$ or $\CG_+\cup\CG_-$.  Let $\Om = \dt\,\dd x^0\w\dd x^1\w\dd x^2\w\dd x^3$ where $\dt$ is defined by \eqref{dt_sign}.  It follows that $*\Om = \frac{1}{\sqrt{(1+\ld_1^2)(1+\ld_2^2)(1+\ld_3^3)}}$; \eqref{Dt_Om} and \eqref{grad_Om} read
\begin{align*}
    \begin{split}
        \Dt(*\Om) &= - (*\Om)\cdot \left[ (\sum_{\af,i,j}h_{\af ij}^2) - 2\ld_1\ld_2\sum_k(h_{41k}h_{52k} - h_{42k}h_{51k}) \right. \\
        &\qquad \left. - 2\ld_2\ld_3\sum_k(h_{52k}h_{63k} - h_{53k}h_{62k}) - 2\ld_3\ld_1\sum_k(h_{41k}h_{63k} - h_{43k}h_{61k}) \right] ~,
    \end{split} \\
    e_k(*\Om) &= - \dt(*\Om)\cdot \sum_{k} (\ld_1h_{41k} + \ld_2h_{52k} + \ld_3h_{63k}) ~.
\end{align*}
Therefore,
\begin{align}
    \Dt(\log(*\Om)^{-1}) &= \frac{-(*\Om)\Dt(*\Om) + |\nabla(*\Om)|^2}{(*\Om)^2} \notag \\
    \begin{split} \label{log_Om_coass}
        &= \sum_{\af,i,j}h_{\af ij}^2 + \sum_k \left( \ld_1^2h_{41k}^2 + \ld_2^2h_{52k} + \ld_3^2h_{63k}^2 \right. \\
        &\qquad\qquad \left. + 2\ld_1\ld_2h_{42k}h_{51k} + 2\ld_2\ld_3h_{53k}h_{62k} + 2\ld_3\ld_1h_{43k}h_{61k} \right)
    \end{split}
\end{align}

According to Lemma \ref{lem_coass_h}, the second fundamental form has only $15$ degrees of freedom, and are grouped as follows.
\begin{description}
    \item[Group 0] $h_{423}$, $h_{531}$ and $h_{612}$.  Note that $h_{401} = h_{612} - h_{531}$, $h_{502} = h_{423}-h_{612}$, and $h_{603} = h_{531} - h_{423}$.
    \item[Group 1] $h_{503}$, $h_{512}$, $h_{602}$ and $h_{631}$.  Note that $h_{400} = -h_{503} + h_{602}$, $h_{411} = -h_{512} - h_{631}$, $h_{422} = h_{512} - h_{602}$, and $h_{433} = h_{503} + h_{631}$.
    \item[Group 2] $h_{601}$, $h_{623}$, $h_{403}$ and $h_{412}$.  Note that $h_{500} = -h_{601} + h_{403}$, $h_{511} = h_{601} + h_{412}$, $h_{522} = -h_{623} - h_{412}$, and $h_{533} = h_{623} - h_{403}$.
    \item[Group 3] $h_{402}$, $h_{431}$, $h_{501}$ and $h_{523}$.  Note that $h_{600} = -h_{402} + h_{501}$, $h_{611} = h_{431} - h_{501}$, $h_{622} = h_{402} + h_{523}$, and $h_{633} = -h_{431} - h_{523}$.
\end{description}
For $i = 0,\ldots,3$, let $\hat{\bh}_i$ be the column vectors:
\begin{align*}
    \hat{\bh}_0 &= (h_{423}, h_{531}, h_{612}) ~,  & \hat{\bh}_1 &= (h_{503}, h_{512}, h_{602}, h_{631}) ~, \\
    \hat{\bh}_2 &= (h_{601}, h_{623}, h_{403}, h_{412}) ~,  & \hat{\bh}_3 &= (h_{402}, h_{431}, h_{501}, h_{523}) ~.
\end{align*}
It follows that \eqref{log_Om_coass} becomes
\begin{align} \label{log_Om_coass1} \begin{split}
    \Dt(\log(*\Om)^{-1}) &= (\hat{\bh}_0)^T \,\bL_0(\ld_1,\ld_2,\ld_3)\, \hat{\bh}_0 + (\hat{\bh}_1)^T \,\bL(\ld_1,\ld_2,\ld_3)\, \hat{\bh}_1 \\
    &\quad + (\hat{\bh}_2)^T \,\bL(\ld_2,\ld_3,\ld_1)\, \hat{\bh}_2 + (\hat{\bh}_3)^T \,\bL(\ld_3,\ld_1,\ld_2)\, \hat{\bh}_3 ~,
\end{split} \end{align}
where
\begin{align} \label{quadr_coass0}
    \bL_0(\ld,\mu,\nu) = \begin{bmatrix}
        6+\mu^2+\nu^2 & -2+\ld\mu-\nu^2 & -2+\nu\ld-\mu^2 \\ -2+\ld\mu-\nu^2 & 6+\nu^2+\ld^2 & -2+\mu\nu-\ld^2 \\ -2+\nu\ld-\mu^2 & -2+\mu\nu-\ld^2 & 6+\ld^2+\mu^2
    \end{bmatrix}
\end{align}
and
\begin{align} \label{quadr_coass1}
    \bL(\ld,\mu,\nu) = \begin{bmatrix}
        4 & 0 & -1+\mu\nu & 1+\nu\ld \\ 0 & 4+(\ld+\mu)^2 & -1-\ld\mu & 1+\ld^2 \\ -1+\mu\nu & -1-\ld\mu & 4 & 0 \\ 1+\nu\ld & 1+\ld^2 & 0 & 4+(\ld+\nu)^2
    \end{bmatrix} ~.
\end{align}

An observation is that the positivity of these matrices on $\CG$ implies the boundedness of $|\ld_i|$'s.
\begin{lem} \label{lem_coass_bdd_ld}
    The subset of $\CG$ where $\bL_0(\ld_1,\ld_2,\ld_3)$, $\bL(\ld_1,\ld_2,\ld_3)$, $\bL(\ld_2,\ld_3,\ld_1)$ and $\bL(\ld_3,\ld_1,\ld_2)$ are all positive semi-definite is bounded.
\end{lem}
\begin{proof}
    If $(\ld_1,\ld_2,\ld_3)\in\CG_0$, it follows from Lemma \ref{lem_SVD_3d} (iii) that $\sm_2 \leq 0$, and from Lemma \ref{lem_sms_3D} (ii) that $3(\sm_1)^2 \leq (\sm_2)^2$.  By using $\sm_1 = \sm_3$,
    \begin{align*}
        0 &\leq \det(\bL_0(\ld_1,\ld_2,\ld_3)) \\
        &= 4 ((\sm_2)^3 - 4(\sm_2)^2 - 8\sm_2 + 32) + (\sm_1)^2 (- (\sm_2)^2 + 6\sm_2 + 31) \\
        &\leq 4 (\sm_2)^3 - (16-\frac{31}{3})(\sm_2)^2 - 32\sm_2 + 128 ~.
    \end{align*}
    It follows that $\sm_2 < -4$, and $(\sm_1)^2 < \frac{16}{3}$.  Thus, $\ld_1^2+\ld_2^2+\ld_3^2 = (\sm_1)^2 - 2\sm_2 < \frac{40}{3}$.

    If $(\ld_1,\ld_2,\ld_3)\in\CG_+$, consider the sub-block of $\bL(\ld_1,\ld_2,\ld_3)$ consisting of the first and third columns-rows.  Its semi-positivity implies that $\ld_1\ld_2 \leq 5$, and \begin{align*}
        \ld_1+\ld_2+\ld_3 = \sqrt{(\ld_1\ld_2)(\ld_2\ld_3)(\ld_3\ld_1)} \leq 5\sqrt{5} ~.
    \end{align*}
    For $\CG_-$, it follows by reflection.
\end{proof}

The above discussion can be summarized as the following proposition.
\begin{prop} \label{prop_coass_Bernstein}
    Let $X\subset\BR^7$ be coassociative submanifold that is entire graphical over $\BR^4\times\{0\}$.  Suppose that its singular values $(\ld_1,\ld_2,\ld_3)$ belongs to a subset
    \begin{itemize}
        \item which is invariant under the action of $S_3\times\{\pm1\}$ (relabeling $\ld_i$'s and flipping the signs),
        \item and over which $\bL_0 \geq \vep\,\bfI_3$ and $\bL \geq \vep\,\bfI_4$ for some $\vep>0$.
    \end{itemize}
    Then, $X$ must be affine.
\end{prop}
\begin{proof}
    The proof is identical to that in \cites{Wang03,TW02}.  It is included for the convenience of the reader.  Notice that being calibrated and the singular values are invariant under scaling.  By Lemma \ref{lem_coass_bdd_ld}, the gradient is bounded, and one can perform blow-down to obtain a coassociative cone $K_{\infty}$.  Suppose that $K_{\infty}$ has singularity only at the origin.  The assumption and \eqref{log_Om_coass1} imply that $\Dt \log(*\Om)^{-1}\geq \vep' |A_{{K}_{\infty}}|^2$ for some $\vep' > 0$ on ${K}_{\infty}\setminus\{\text{origin}\}$.  By the maximum principle, ${K}_{\infty}$ is flat.  It follows from Allard's Regularity Theorem \cite{A72} that $X$ is affine.
    
    If ${K}_{\infty}$ has singularity other than the origin, one can perform blow-up and apply Federer's dimension reduction \cite{F70}, \cite{F69}*{ch.4.3 and ch.5.4} to get a minimal cone $K$ of smaller dimension.  By repeating the procedure if necessary, the cone $K$ has singularity only at the origin.  It is graphical over some $\BR^k$, and let $\td{\Om}$ be the parallel transport of the volume form of $\BR^k$.  It is not hard to derive from the condition of this proposition that on $K$, $\Dt \log(*\td{\Om})^{-1}\geq \vep' |A_K|^2$.  Thus, $K$ is flat, and ${K}_{\infty}$ has singularity only at the origin.

    In fact, $K$ is a {special Lagrangian} cone.  For instance, suppose that $K$ is $3$-dimensional, $K\times\BR u$ is coassociative for some unit vector $u\perp K$, and $K\times\BR u$ is graphical over $\BR^4\times\{0\}$.  Lemma \ref{lem_coass_SVD} leads to the following dichotomy: after a $G_2$ change of coordinate, $u$ is either $\frac{\pl}{\pl x^0}$, or $\frac{1}{\sqrt{1+\ld_3^2}}(\frac{\pl}{\pl x^3} + \ld_3\frac{\pl}{\pl y^3})$ for some constant $\ld_3\neq0$.  If $u = \frac{\pl}{\pl x^0}$, $K$ is calibrated by
    \begin{align*}
        \iota(\frac{\pl}{\pl x^0})\psi &= \dd x^{123} - \dd y^{23}\w\dd x^1 - \dd y^{31}\w\dd x^2 - \dd y^{12}\w\dd x^3 \\
        &= \re\left[(\dd x^1+\sqrt{-1}\dd y^1)\w(\dd x^2+\sqrt{-1}\dd y^2)\w(\dd x^3+\sqrt{-1}\dd y^3)\right] ~.
    \end{align*}
    In the second case, the projections of $u$ onto $\BR^4\times\{0\}$ and $\{0\}\times\BR^3$ pick up $\frac{\pl}{\pl x^3}$ and $\frac{\pl}{\pl y^3}$, and their cross product picks up $\frac{\pl}{\pl x^0}$.  It follows that $K = K'\times\BR\frac{\pl}{\pl x^0}$, and $K'$ is calibrated by
    \begin{align*}
        \iota(\frac{\pl}{\pl x^0})\iota(u)\psi &= \re\left[e^{\sqrt{-1}(\pi+\arctan\ld_3)}(\dd x^1+\sqrt{-1}\dd y^1)\w(\dd x^2+\sqrt{-1}\dd y^2))\right] ~.
    \end{align*}
    In either case, the calculation of $\log(*\td{\Om})^{-1}$ is derived in \cite{TW02}, and is formally the reduction of \eqref{log_Om_coass1}.
\end{proof}

\subsection{Bernstein Polynomials and a Bernstein Theorem}

The main purpose of this section is to identify a specific region where the assumption of Proposition \ref{prop_coass_Bernstein} holds true.  The verification will be reduced to show that a single variable polynomial is positive or non-negative on a closed interval.  There are many ways to do it, and we will present it in terms of a theorem of Bernstein.

Recall that the $m+1$ Bernstein basis polynomials of degree $m$ are defined as
\begin{align} \label{Bernstein_poly}
    \fb_{k,m}(t) &= \binom{m}{k} t^k(1-t)^{m-k}
\end{align}
for $k = 0,\ldots,m$.  Any polynomial of degree no greater than $m$ can be represented uniquely as their linear combination:
\begin{align*}
    \fp(t) &= \sum_{k=0}^m c_{k,m}\,\fb_{k,m}(t) ~.
\end{align*}
A theorem of Bernstein says that a polynomial $\fp(t)$ is non-negative on $[0,1]$ if and only if there exists some $m>\!>1$ such that $c_{k,m} \geq 0$ for all $k$.  It follows that the positivity on $[0,1]$ corresponds to $c_{0,m}>0$, $c_{m,m}>0$, and $c_{k,m} \geq 0$ for all $k$.  This theorem is not constructive in the sense that $m$ may have be chosen to be quite large.  See \cite{PR00} and the references therein for the discussions about this theorem, and an upper bound on $m$.

\begin{defn}
    A polynomial $\fp(t)$ is said to be \emph{positive on $[a,b]$ of Bernstein degree $m$} if
    \begin{align*}
        \fp(t) &= \sum_{k=0}^m c_{k,m}\,\fb_{k,m}(\frac{t-a}{b-a})
\end{align*}
with $c_{0,m}>0$, $c_{m,m}>0$, and $c_{k,m} \geq 0$ for all $k$.
\end{defn}
For example,
\begin{align*}
    5 t^2 - 3 t + 1 &= 3 t^2 - t (1 - t) + (1 - t)^2 \\
    &= 3 t^3 + 2 t^2 (1 - t) + (1 - t)^3 
\end{align*}
is positive on $[0,1]$ of Bernstein degree $3$, but not of Bernstein degree $2$.  Here is a polynomial which we will encounter momentarily:
\begin{align}
    t^4 - 2t^3 - 12t^2 - 16t + 20 &= \frac{1520}{2401} (\oh - t)^4 + \frac{144}{2401} (\oh - t)^3 (3 + t) + \frac{3072}{2401} (\oh - t)^2 (3 + t)^2 \notag \\
    &\quad + \frac{2188}{2401} (\oh - t) (3 + t)^3 + \frac{141}{2401} (3 + t)^4 ~; \label{Bernstein_poly_eg}
\end{align}
it is positive on $[-3,\oh]$ of Bernstein degree $4$.  
With $[a,b]$ and $m$, one only has to solve a rank $m+1$ linear system for the coefficients.  The coefficients are by no means enlightening, and will not be written down in the rest of this paper.

We are now ready to prove the main result of this section.
\begin{thm} \label{thm_Bernstein_coass}
    Let $X\subset\BR^7$ be coassociative submanifold that is entire graphical over $\BR^4\times\{0\}$.  Suppose either one of the following holds:
    \begin{enumerate}
        \item Its singular values $(\ld_1,\ld_2,\ld_3)$ belong $\CG_0$ and satisfy $\ld_1\ld_2+\ld_2\ld_3+\ld_3\ld_1\geq -2\sqrt{2} + \vep$ for some $0<\vep<\!<1$,
        \item Let $\tau$ be the largest real root\footnote{It has four real roots and two complex roots.  Its discriminant is negative.  Its roots lie within $(4,4\oh)$, $(2,2\oh)$, $(1\oh,2)$ and $(-3,-2\oh)$.} of $t^6 - 6 t^5 + 6 t^4 + 16 t^3 - 77 t^2 + 204 t - 192$.  The singular values of $X$ belongs $\CG_+\cup\CG_-$ and satisfy $\ld_i\ld_j\leq \tau - \vep$ for any $i\neq j$ (for some $0<\vep<\!<1$).
    \end{enumerate}
    Then, $X$ must be affine.
    In particular, if $1<\ld_i\ld_j\leq 4$, $X$ is affine.
\end{thm}
\begin{proof}
    \textit{Step 1: the strategy}. Denote the subset of $\CG$ defined by these conditions by $\CP_0^\vep$, $\CP_+^
    \vep$ and $\CP_-^\vep$, respectively.  It is not hard to see that these sets are connected and bounded.  The set $\CP_0^\vep$ contains $(0,0,0)$, at which both $\bL_0$ and $\bL$ are positive definite.  Suppose that $\det(\bL_0)$ and $\det(\bL)$ are both greater than $\vep'$ on $\CP_0^\vep$, for some $\vep'>0$.  It follows from the continuity of eigenvalues and connectedness of $\CP_0^\vep$ that $\bL_0$ and $\bL$ are positive definite on $\CP_0^\vep$.  The boundedness of $\CP_0^\vep$ gives a uniform upper bound of the maximum eigenvalues of $\bL_0$ and $\bL$.  It together with the lower bound of their determinants leads to a uniform lower bound on the minimum eigenvalues.  Hence, $\bL_0$ and $\bL$ are both strictly positive definite on $\CP_0^\vep$, and this theorem follows from Proposition \ref{prop_coass_Bernstein}.

    The arguments for $\CP_\pm^\vep$ are the same; one uses $\pm(\sqrt{3},\sqrt{3},\sqrt{3})$ instead of $(0,0,0)$.  Therefore, it remains to show that both $\det(\bL_0)$ and $\det(\bL)$ have positive lower bounds on $\CP_0^\vep$ and $\CP_\pm^\vep$.
    
    \textit{Step 2: $\det(\bL_0)$ on $\CP_0^\vep$}.  By using $\sm_1 = \sm_3$,
    \begin{align}
        \det(\bL_0) &= 4(2\sqrt{2} + \sm_2)(2\sqrt{2} - \sm_2)(4 - \sm_2) + \sm_1^2 (40 - (3-\sm_2)^2) ~.  \label{eqn_det_bL0_1}
    \end{align}
    Due to Lemma \ref{lem_SVD_3d} (iii), $0\geq\sm_2\geq -2\sqrt{2}+\vep$.  It follows that \eqref{eqn_det_bL0_1} has a strictly positive lower bound.

    \textit{Step 3: $\det(\bL)$ on $\CP_0^\vep$}.  Let
    \begin{align*}
        s &= -\sm_2 = -\ld_1\ld_2 - \frac{(\ld_1+\ld_2)^2}{\ld_1\ld_2-1} ~,
    \end{align*}
    and $t = \ld_1\ld_2$.  They satisfy $0\leq s\leq 2\sqrt{2}-\vep$ and $t<1$ on $\CP_0^\vep$.  It follows from $(\ld_1+\ld_2)^2 \geq 0$ that $s+t \geq 0$, and it follows from $(\ld_1-\ld_2)^2 \geq 0$ that $(1-t)s \geq t(t+3)$.  In particular, on $\CP_0^\vep$,
    \begin{align}
       0\leq s\leq 2\sqrt{2}-\vep ~,~ s+t\geq0 ~,~ t\leq\oh ~.  \label{domain_bL_0}
    \end{align}

    We compute
    \begin{align}
        &\quad \det(\bL(\frac{\ld_1+\ld_2}{\ld_1\ld_2-1},\ld_1,\ld_2))\cdot(1-\ld_1\ld_2)^2 - 4(1-t)s(8-s^2) \notag \\
        \begin{split} \label{eqn_det_bL_1}
        &= s^2 (- t^4 + 2t^3 + 12 t^2 + 16 t -20) + 2 s (2 t^5 - 7 t^4 - 24 t^3 + 68 t^2 - 42 t + 12) \\
        &\quad + (4 t^6 - 4 t^5 - 93 t^4 + 78 t^3 + 240 t^2 - 408 t + 192) ~.
        \end{split}
    \end{align}
    Note that $4(1-t)s(8-s^2)$ is non-negative subject to \eqref{domain_bL_0}.  By subtracting this term, the right hand side of \eqref{eqn_det_bL_1} is quadratic in $s$.

    Denote the right hand side of \eqref{eqn_det_bL_1} by $\fl_1(s,t)$.  We claim that $\fl_1(s,t)$ is non-negative on $D = \{(s,t)\in\BR^2: 0\leq s\leq 2\sqrt{2} , s+t\geq0 , \text{ and } t\leq\oh \}$.  Note that $t\in[-2\sqrt{2},\oh]\subset[-3,\oh]$.  By \eqref{Bernstein_poly_eg}, $- t^4 + 2 t^3 + 12 t^2 + 16 t - 20$ is negative over $[-3,\oh]$.  That is to say, the coefficient of $s^2$ in $\fl_1(s,t)$ is always negative.  It follows that for every $t\in[-2\sqrt{2},\oh]$, the minimum of $\fl_1(s,t)$ must be achieved by $s=0$, $s=2\sqrt{2}$, or $s=-t$.  With this understood, the non-negativity of $\fl_1(s,t)$ on $D$ follows from the following facts:
    \begin{itemize}
        \item $\fl_1(0,t) = 4 t^6 - 4 t^5 - 93 t^4 + 78 t^3 + 240 t^2 - 408 t + 192$ is positive on $[0,\oh]$ of Bernstein degree $6$.
        \item $\fl_1(s,-s) = (1 + s) (2 + s) (2\sqrt{2} + s) (2\sqrt{2} - s) (s^2 + 9 s + 12)$ is non-negative on $[0,2\sqrt{2}]$.
        \item $\fl_1(2\sqrt{2},t) = (2\sqrt{2}+t) \big[ 4 t^5 - 4 t^4 - (101 + 20\sqrt{2}) t^3 + (174 + 106\sqrt{2}) t^2 - (88 + 76\sqrt{2}) t + (24 + 8\sqrt{2}) \big]$.  The polynomial\footnote{Its Bernstein degree on $[-2\sqrt{2},\oh]$ is much higher.} $\fl_1(2\sqrt{2},t)/(2\sqrt{2}+t)$ is positive on $[0,\oh]$ of Bernstein degree $13$, and is positive on $[-2\sqrt{2},0]$ of Bernstein degree $5$.
    \end{itemize}

    These discussions imply that $\det(\bL)$ has a positive lower bound on $\CP_0^\vep$.

    \textit{Step 4: $\det(\bL_0)$ on $\CP_{\pm}^\vep$}.  It suffices to show that $\det(\bL_0)$ has a positive lower bound when $9 \leq \sm_2 \leq 15$.  Due to \eqref{eqn_det_bL0_1}, it is obvious if $\sm_2 \leq 3 + \sqrt{10}$.
    
    If $\sm_2 > 3 + \sqrt{10}$, $40 - (3-\sm_2)^2 < 0$.  By \eqref{eqn_det_bL0_1} and Lemma \ref{lem_sms_3D} (ii),
    \begin{align*}
        \det(\bL_0) &\geq 4(2\sqrt{2} + \sm_2)(2\sqrt{2} - \sm_2)(4 - \sm_2) + \frac{1}{3}\sm_2^2 (40 - (3-\sm_2)^2) \\
        &= \frac{1}{3} (384 - 96 \sm_2 - 17 \sm_2^2 + 18 \sm_2^3 - \sm_2^4) ~,
    \end{align*}
    which is positive on $[9,15]$ of Bernstein degree $4$.

    \textit{Step 5: $\det(\bL)$ on $\CP_{\pm}^\vep$}.  As in step 3, let $t = \ld_1\ld_2 \in (1,\tau-\vep]$.  Let $u = (\ld_1 - \ld_2)^2$.  By a direct computation,
    \begin{align}
        &\quad \det(\bL(\frac{\ld_1+\ld_2}{\ld_1\ld_2-1},\ld_1,\ld_2))\cdot(1-\ld_1\ld_2)^4  \notag \\
        \begin{split} \label{eqn_det_bL_3}
        &= - 4 u^3 - u^2 (t^4 - 2 t^3 + 20 t + 20) \\
        &\quad + u ( - 6 t^6 + 16 t^5 + 58 t^4 - 152 t^3 + 200 t^2 - 292 t + 56) \\
        &\qquad + ( - t^8 - 10 t^7 + 18 t^6 + 248 t^5 - 233 t^4 - 310 t^3 + 608 t^2 - 624 t + 192) ~.
        \end{split}
    \end{align}
    
    It follows from $\ld_1\ld_3 \leq \tau - \vep$ and $\ld_2\ld_3 \leq \tau - \vep$ that
    \begin{align}
        \max\{\ld_1^2,\ld_2^2\} + \ld_1\ld_2 &\leq (\tau-\vep)(\ld_1\ld_2 - 1)  \notag \\
        \Leftrightarrow\quad  \sqrt{u(s+4t)} &\leq 2((\tau-\vep-2)t - (\tau-\vep)) - u  \label{domain_bL_2}
    \end{align}
    where we have used $2\max\{\ld_1^2,\ld_2^2\} = \sqrt{u} + \sqrt{u+4t}$.  We divide it into two cases.
    \begin{enumerate}
        \item[(a)] When $t\leq 4$, a less sharp form of \eqref{domain_bL_2} will be enough.  Note that the left hand side of \eqref{domain_bL_2} is greater than $u$, and the right hand side is no greater than $-u + 2(\tau - \vep)(t-1) - 2t$.  Since $\tau < 4\oh$, one finds that
        \begin{align} \label{domain_bL_3}
            0 \leq u < \frac{5}{2}t - \frac{9}{2} ~, \quad
            \text{and in particular }~ \frac{9}{5} < t \leq 4 ~.
        \end{align}

        \item[(b)] When $4 < t \leq \tau-\vep$, squaring both sides of \eqref{domain_bL_2} leads to
        \begin{align} \label{domain_bL_4}
            0 \leq u &\leq \frac{((\tau-\vep)(t-1) - 2t)^2}{(\tau-\vep)(t-1) - t} ~. 
        \end{align}
        Here, we implicitly assume that $\vep < \frac{1}{10}$.
    \end{enumerate}

    \textit{Step 5--case (a)}.  This case will be treated by the same framework as step 3.  Under the condition \eqref{domain_bL_3}, $u^3 < \frac{1}{8}(5t-9)^3$, and \eqref{eqn_det_bL_3} is greater than
    \begin{align*}
        \fl_2(u,t) &= -u^2 (t^4 - 2 t^3 + 20 t + 20) + u ( - 6 t^6 + 16 t^5 + 58 t^4 - 152 t^3 + 200 t^2 - 292 t + 56) \\
        &\quad + \oh ( - 2 t^8 - 20 t^7 + 36 t^6 + 496 t^5 - 466 t^4 - 745 t^3 + 1891 t^2 - 2463 t + 1113) ~.
    \end{align*}
    Since $t^4 - 2 t^3 + 20 t + 20$ is positive on $[1,5]$ of Bernstein degree $4$, the minimum of $\fl_2(u,t)$ subject to (the closure of) \eqref{domain_bL_3} must be attained by $u = 0$ or $u = \oh(5t - 9)$.  Since $\fl_2(0,t)$ is positive on $[\frac{3}{2},4]$ of Bernstein degree $8$, and $\fl_2(\oh(5t-9),t)$ is positive on $[1,4]$ of Bernstein degree $8$, $\fl_2(u,t)$ has a positive lower bound subject to \eqref{domain_bL_3}.

    \textit{Step 5--case (b)}.  We first examine the coefficients of $u^2$, $u$ and $1$ in \eqref{eqn_det_bL_3}.  Multiply these three polynomials by $-1$.  They have the following properties for $t\in[4,5]$.
    \begin{itemize}
        \item quadratic: $t^4 - 2 t^3 + 20 t + 20$ is positive on $[4,5]$ of Bernstein degree $4$, and its derivative is positive on $[4,5]$ of Bernstein degree $3$;
        \item linear: $6 t^6 - 16 t^5 - 58 t^4 + 152 t^3 - 200 t^2 + 292 t - 56$ is positive on $[4,5]$ of Bernstein degree $6$, and its derivative is positive on $[4,5]$ of Bernstein degree $5$;
        \item constant: the derivative of $t^8 + 10 t^7 - 18 t^6 - 248 t^5 + 233 t^4 + 310 t^3 - 608 t^2 + 624 t - 192$ is positive on $[4,5]$ of Bernstein degree $7$.
    \end{itemize}
    Also, note that the derivative in $t$ of the upper bound in \eqref{domain_bL_4} is
    \begin{align*}
        \frac{(\tau-\vep)(t-1)-2t}{((\tau-\vep)(t-1)-t)^2} \left[ (\tau-\vep)(\tau-\vep-3)(t-1) + 2t \right] &> 0 ~.
    \end{align*}
    It follows that the minimum of \eqref{eqn_det_bL_3} subject to \eqref{domain_bL_4} and $4 < t \leq \tau-\vep$ is achieved at
    \begin{align*}
        u = \frac{(\tau-\vep)(\tau-\vep-3)^2}{\tau-\vep-2}  \quad\text{and}\quad  t = \tau - \vep ~.
    \end{align*}
    The value of \eqref{eqn_det_bL_3} at this point is
    \begin{align*}
        & - \frac{8(\tau - \vep - 1)^5}{(\tau - \vep - 2)^3} \left[ (\tau - \vep)^6 - 6 (\tau - \vep)^5 + 6 (\tau - \vep)^4 \right. \\
        &\qquad\qquad\qquad\qquad \left. + 16 (\tau - \vep)^3 - 77 (\tau - \vep)^2 + 204 (\tau- \vep) - 192 \right]
    \end{align*}
    which is positive due to the definition of $\tau$.
\end{proof}




The regions given by Theorem \ref{thm_Bernstein_coass} with $\vep = \frac{1}{10}$ are plotted in Figure \ref{coass_inner} and \ref{coass_outer}.  The dotted curves are the zero loci on the $(\ld_1,\ld_2)$-plane of $\det(\bL_0(\ld_1,\ld_2,\ld_3))$, $\det(\bL(\ld_1,\ld_2,\ld_3))$, $\det(\bL(\ld_2,\ld_3,\ld_1))$ and $\det(\bL(\ld_3,\ld_1,\ld_2))$ subject to $\ld_1+\ld_2+\ld_3 = \ld_1\ld_2\ld_3$.  The solid curve in Figure \ref{coass_inner} is the level set $\ld_1\ld_2 + \ld_2\ld_3 + \ld_3\ld_1 = -2\sqrt{2} + \vep$; the solid curves in Figure \ref{coass_outer} are the level sets $\ld_i\ld_j = \tau - \vep$ for $i\neq j$.

\begin{figure}
\centering
    \begin{minipage}[b]{0.35\textwidth}
    \includegraphics[width=\textwidth]{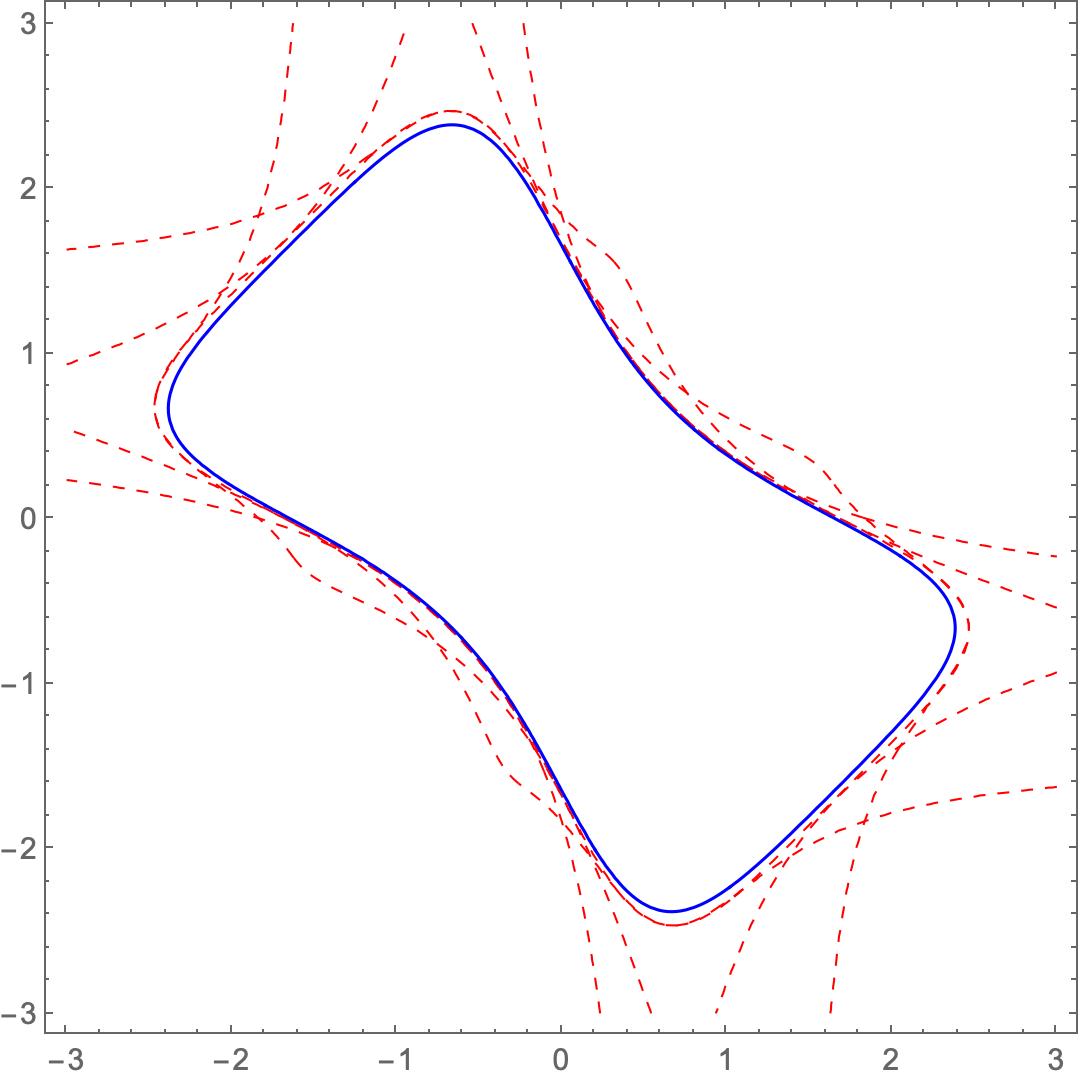}
    \caption{On $\CG_0$}
    \label{coass_inner}
    \end{minipage}
\hspace{1cm}
    \begin{minipage}[b]{0.35\textwidth}
    \includegraphics[width=\textwidth]{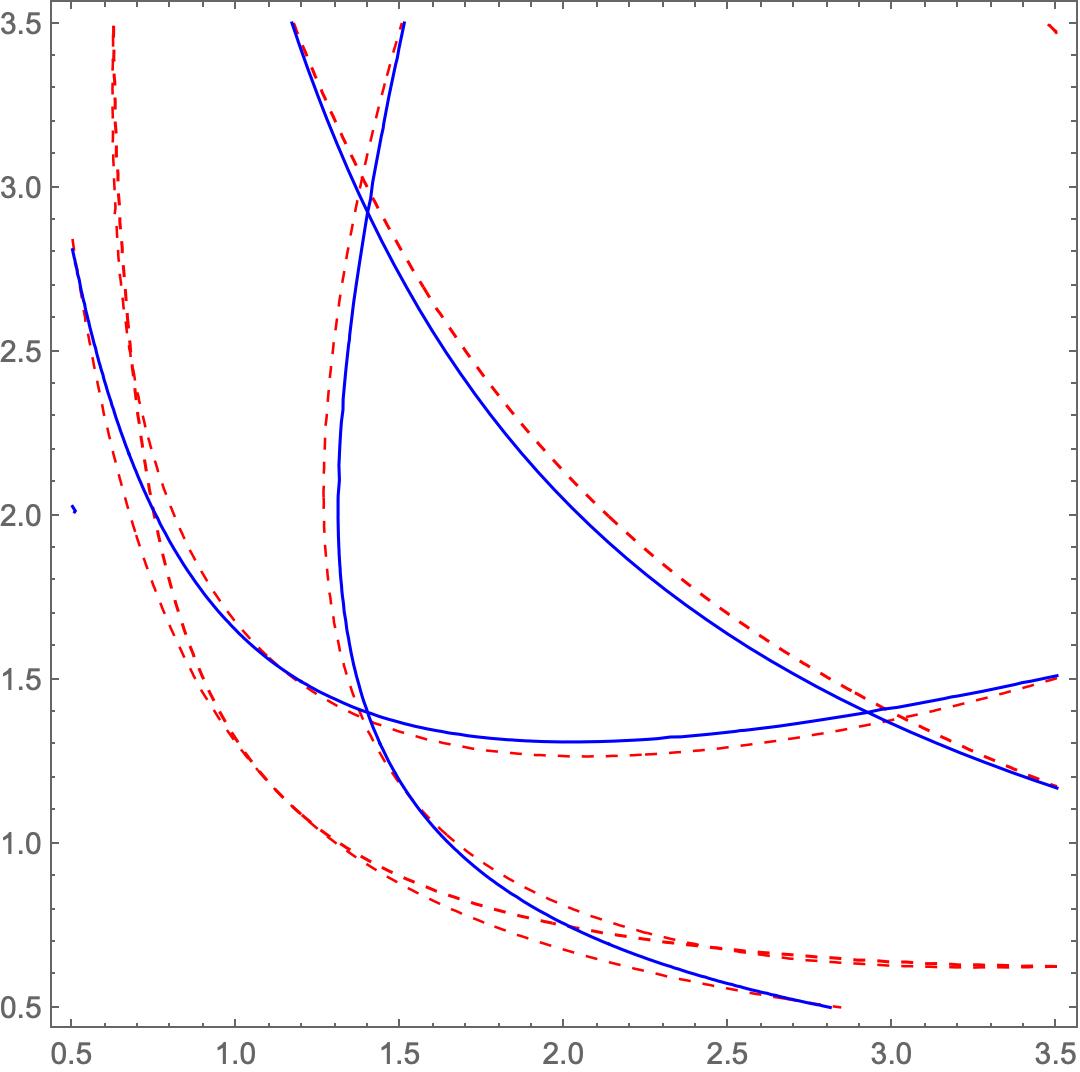}
    \caption{On $\CG_+$}
    \label{coass_outer}
    \end{minipage}
\end{figure}

\begin{rmk}    
    In fact, if $\bL \geq 0$ on a subset of $\CG_+\cup\CG_-$ which is invariant under the action of $S_3\times\{\pm1\}$, then $\bL_0 \geq 0$.  But this assertion requires more work than step 4.
\end{rmk}

\section{Bernstein theorem for Cayley Graphs} \label{sec_Bernstein_Cayley}

Let $Z\subset\BR^8$ be a Cayley submanifold.  Suppose that it is graphical over $\BR^4\times\{0\}$, and its singular values given by Lemma \ref{lem_Cayley_SVD} belongs to the component $\CS_0$ described by Lemma \ref{lem_SVD_4d}.

Set $\Om$ to be $\dd x^0\w\dd x^1\w\dd x^2\w\dd x^3$.  We can evaluate the right hand side of \eqref{Dt_Om} and \eqref{grad_Om} at any point by using the corresponding $\spin(7)$ basis produced by Corollary \ref{cor_Cayley_SVD}.  By a direct computation,
\begin{align} \begin{split} \label{log_Om_Cayley}
    \Dt(\log(*\Om)^{-1}) &= \sum_{\af,i,j}h_{\af ij}^2 + \sum_k \big( \ld_0^2h_{40k}^2 + \ld_1^2h_{51k}^2 + \ld_2^2h_{62k} + \ld_3^2h_{73k}^2 \\
    &\qquad\qquad\qquad + 2\ld_0\ld_1h_{41k}h_{50k} + 2\ld_0\ld_2h_{42k}h_{60k} + 2\ld_0\ld_3h_{43k}h_{70k}  \\
    &\qquad\qquad\qquad\quad + 2\ld_1\ld_2h_{52k}h_{61k} + 2\ld_1\ld_3h_{53k}h_{71k} + 2\ld_2\ld_3h_{63k}h_{72k}\big)
    \end{split}
\end{align}
where $0\leq i,j,k \leq 3$ and $4\leq \af\leq 7$.

Due to Lemma \ref{lem_Cayley_h}, the second fundamental form has only $24$ degrees of freedom, and can be grouped as follows.
\begin{description}
    \item[Group 4] $h_{501}$, $h_{523}$, $h_{602}$, $h_{613}$, $h_{703}$ and $h_{712}$.  For $i=0,1,2,3$, $h_{4ii}$ can be expressed as their linear combinations.
    \item[Group 5] $h_{401}$, $h_{423}$, $h_{603}$, $h_{612}$, $h_{702}$ and $h_{713}$.  For $i=0,1,2,3$, $h_{5ii}$ can be expressed as their linear combinations.
    \item[Group 6] $h_{402}$, $h_{413}$, $h_{503}$, $h_{512}$, $h_{701}$ and $h_{723}$.  For $i=0,1,2,3$, $h_{6ii}$ can be expressed as their linear combinations.
    \item[Group 7] $h_{403}$, $h_{412}$, $h_{502}$, $h_{513}$, $h_{601}$ and $h_{623}$.  For $i=0,1,2,3$, $h_{7ii}$ can be expressed as their linear combinations.
\end{description}
For $\af = 4,\ldots,7$, let $\td{\bh}_\af$ be the column vectors:
\begin{align*}
    \check{\bh}_4 &= (h_{523}, h_{613}, h_{712}, h_{501}, h_{602}, h_{703}) ~,  & \check{\bh}_5 &= (h_{603}, h_{702}, h_{423}, -h_{612}, -h_{713}, -h_{401}) ~, \\
    \check{\bh}_6 &= (h_{701}, h_{413}, h_{503}, h_{723}, h_{402}, h_{512}) ~,  & \check{\bh}_7 &= (-h_{412}, -h_{502}, -h_{601}, h_{403}, h_{513}, h_{623}) ~.
\end{align*}
After a straightforward calculation, \eqref{log_Om_Cayley} becomes
\begin{align} \label{log_Om_Cayley1} \begin{split}
    \Dt(\log(*\Om)^{-1}) &= (\check{\bh}_4)^T \,\bM(\ld_0,\ld_1,\ld_2,\ld_3)\, \check{\bh}_4 + (\check{\bh}_5)^T \,\bM(\ld_1,\ld_2,\ld_3,\ld_0)\, \check{\bh}_5 \\
    &\quad + (\check{\bh}_6)^T \,\bM(\ld_2,\ld_3,\ld_0,\ld_1)\, \check{\bh}_6 + (\check{\bh}_7)^T \,\bM(\ld_3,\ld_0,\ld_1,\ld_2)\, \check{\bh}_7 ~,
\end{split} \end{align}
where
\begin{align} \label{quadr_Cayley}
    \bM(\eta, \ld, \mu, \nu) = \begin{bmatrix}
        4 & - 1 + \ld\mu & - 1 + \nu\ld & 0 & - 1 - \eta\mu & 1 + \eta\nu \\
        - 1 + \ld\mu & 4 & - 1 + \mu\nu & 1 + \eta\ld & 0 & - 1 - \eta\nu \\
        - 1 + \nu\ld & - 1 + \mu\nu & 4 & - 1- \eta\ld & 1 + \eta\mu & 0 \\
        0 & 1 + \eta\ld & - 1 - \eta\ld & 4 + (\eta + \ld)^2 & 1 + \eta^2 & 1 + \eta^2 \\
        - 1 - \eta\mu & 0 & 1 + \eta\mu & 1 + \eta^2 & 4 + (\eta + \mu)^2 & 1 + \eta^2 \\
        1 + \eta\nu & - 1 - \eta\nu & 0 & 1 + \eta^2 & 1 + \eta^2 & 4 + (\eta + \nu)^2
    \end{bmatrix} ~.
\end{align}

We are now ready to prove the main result of this section.
\begin{thm} \label{thm_Bernstein_Cayley}
    Let $Z\subset\BR^8$ be Cayley submanifold which is entire graphical over $\BR^4\times\{0\}$.  Suppose that there exists some $\vep > 0$ such that its singular values satisfy
    \begin{align} \label{condition_Bernstein_Cayley}
        -\sqrt{6}+\vep \leq \ld_0\ld_1 + \ld_0\ld_2 + \ld_0\ld_3 + \ld_1\ld_2 + \ld_1\ld_3 + \ld_2\ld_3 \leq 0 ~.
    \end{align}
    Then, $Z$ must be affine.
\end{thm}
\begin{proof}
    The strategy of the proof is similar to that in Section \ref{sec_Bernstein_coass}.  The main task is to show that \eqref{condition_Bernstein_Cayley} implies the boundedness of $(\ld_0,\ld_1,\ld_2,\ld_3)$, and the positivity of $\det(\bM(\ld_0,\ld_1,\ld_2,\ld_3))$.

    \textit{Step 1: boundedness of $\bld$}.  Let $s_1 = \ld_1+\ld_2+\ld_3$, $s_2 = \ld_1\ld_2 + \ld_2\ld_3 + \ld_3\ld_1$, $s_3 = \ld_1\ld_2\ld_3$, and $w = \ld_0$.  By Lemma \ref{lem_Cayley_SVD}, $s_3 = s_1 + w(1 - s_2)$.  Lemma \ref{lem_SVD_4d} and \eqref{condition_Bernstein_Cayley} implies that $1 - s_2 > 0$.  One infers from \eqref{condition_Bernstein_Cayley} that
    \begin{align*}
        0 \leq (s_2 + ws_1 + \sqrt{6}) (1 - s_2) &= (\sqrt{6} + s_2)(1 - s_2) + s_1(s_3 - s_1) \\
        &\leq (\sqrt{6} + s_2)(1 - s_2) + \frac{1}{3}s_2^2 - s_1^2 ~,
    \end{align*}
    where the last inequality uses Lemma \ref{lem_sms_3D} (ii).  It follows that both $|s_1|$ and $|s_2|$ are bounded.  Thus, $\ld_1^2+\ld_2^2+\ld_3^2 = s_1^2 - 2s_2$ is also bounded.  Since \eqref{condition_Bernstein_Cayley} is symmetric, $|w| = |\ld_0|$ is bounded as well.

    \textit{Step 2: positivity of $\det(\bN(\ld_0,\ld_1,\ld_2,\ld_3))$}.  The first step is to express relevant functions in terms of $s_1$, $s_2$ and $t$.  Note that $\sm_2(\ld_0,\ld_1,\ld_2,\ld_3)$ is $s_2 + ws_1$.  We will then study the minimum of $\det(\bN(\ld_0,\ld_1,\ld_2,\ld_3))$ under the constraint $s_2 + ws_1 = -\vsm$ for $\vsm\in[0,\sqrt{6})$.

    It is easy to see that $\det(\bN(\ld_0,\ld_1,\ld_2,\ld_3))$ is symmetric in $\ld_1$, $\ld_2$, $\ld_3$, and thus can be regarded as a function in $s_1$, $s_2$, $s_3$ and $w$.  With $s_3 = s_1 + w(1 - s_2)$, this determinant is a function in $s_1$, $s_2$ and $w$, and denote it by $4\cdot\ff(s_1,s_2,w)$.  By a direct computation,
    \begin{align} \begin{split}
        &\quad \ff(s_1, s_2, w) \\
        &= 14(w^2 + 1)s_1^4 + 4(-5 w^3 + 2 w)s_1^3 s_2 + (8 w^4 - 39 w^2 + 1) s_1^2 s_2^2 + 8(2 w^3 -3 w) s_1 s_2^3 \\
        &\quad + 4(w^2 - 1) s_2^4 + 4(9 w^3 + 2w) s_1^3 + (-39 w^4 + 76 w^2 - 41) s_1^2 s_2 \\
        &\quad + 2(2 w^5 - 49 w^3 + 57 w) s_1 s_2^2 + 4(w^6 + 6 w^4 - 14 w^2 + 9) s_2^3 \\
        &\quad + (-17 w^4 + 39 w^2 + 164) s_1^2 + (w^5 -32 w^3 - 62 w) s_1 s_2 + 3(3 w^4 + 19 w^2 - 16) s_2^2 \\
        &\quad + 2(- 7 w^5 + 57 w^3 -10 w) s_1 + (- 3 w^6 - 8 w^4 + 207 w^2 - 216) s_2 \\
        &\quad + (w^6 + 21 w^4 + 32 w^2 + 432) ~.
    \end{split} \label{fct_Cayley_f} \end{align}
    Note that $(\ld_1-\ld_2)^2(\ld_2-\ld_3)^2(\ld_3-\ld_1)^2$ subject to $s_3 = s_1 + w(1 - s_2)$ is also a function in $s_1$, $s_2$ and $w$.  Call this function $\fg(s_1,s_2,w)$:
    \begin{align} \begin{split}
        \fg(s_1, s_2, w)&= -4 s_1^4 + 4w s_1^3 s_2 + s_1^2 s_2^2 - 4w s_1^3 + 18 s_1^2 s_2 - 18w s_1 s_2^2 - 4 s_2^3 \\
        &\quad - 27 s_1^2 + 72w s_1 s_2 - 27 w^2 s_2^2 - 54w s_1 + 54 w^2 s_2 - 27 w^2 ~.
    \end{split} \label{fct_Cayley_g} \end{align}
    For this theorem, one only needs to analyze $\ff(s_1, s_2, w)$ on where $\fg(s_1, s_2, w) \geq 0$.

    \begin{claim} \label{claim_Cayley}
        For any $\vsm\in[0,\sqrt{6})$, $\ff(s_1, -\vsm - ws_1, w)$ is positive on where $\fg(s_1, -\vsm - ws_1, w) \geq 0$.
    \end{claim}
    
    By step 1, the subset of $\CS_0$ satisfying \eqref{condition_Bernstein_Cayley} is compact.  Together with Claim \ref{claim_Cayley}, $\det(\bM(\bld))$ has a positive lower bound subject to \eqref{condition_Bernstein_Cayley}, and this theorem follows.

    The proof of Claim \ref{claim_Cayley} relies on the classical theory for quadratic polynomials, and is in Appendix \ref{sec_claim}.  Note that $\ff(0,-\vsm,0) = 4 (3 + \vsm) (6 + \vsm) (6 - \vsm^2)$, and thus $\vsm < \sqrt{6}$ is a sharp bound.
\end{proof}

\begin{rmk}
    Suppose that $K$ is $3$-dimensional, $K\times\BR u$ is Cayley for some unit vector $u\perp K$, and $K\times\BR u$ is graphical over $\BR^4\times\{0\}$.  According to Lemma \ref{lem_Cayley_SVD}, we may assume $u$ is $\frac{1}{\sqrt{1+\ld_0^2}}(\frac{\pl}{\pl x^0} + \ld_0\frac{\pl}{\pl y^0})$ after a $\spin(7)$ change of coordinate.
    If $\ld_0 = 0$, $K$ is calibrated by
    \begin{align*}
        \dd x^{123} - \dd x^{1}\w\gm^1 - \dd x^{2}\w\gm^2 - \dd x^{3}\w\gm^3 ~,
    \end{align*}
    and is associative.
    If $\ld_0\neq0$, the projections of $u$ onto $\BR^4\times\{0\}$ and $\{0\}\times\BR^4$ pick up $\frac{\pl}{\pl x^0}$ and $\frac{\pl}{\pl y^0}$, and $K$ belongs to the orthogonal complement of these two coordinate directions, and is calibrated by
    \begin{align*}
        \left.\iota(u)\Phi\right|_{x^0 = 0 = y^0}
        &= \re\left[e^{\sqrt{-1}\arctan\ld_0}(\dd x^1+\sqrt{-1}\dd y^1)\w(\dd x^2+\sqrt{-1}\dd y^2))\w(\dd x^3+\sqrt{-1}\dd y^3)\right] ~.
    \end{align*}
    That is to say, $K$ is special Lagrangian.
\end{rmk}

For coassociative submanifolds, we establish Bernstein conditions for both $\CG_0$ and $\CG_+\cup\CG_-$. However, in the Cayley case, we were unable to identify a suitable condition for $\CS_+\cup\CS_-$.  Computer graphics suggests that it is a ``{small}" neighborhood of $\bld=(1,1,1,1)$.  It is interesting to establish a Bernstein theorem for this component.

\appendix
\section{The Proof of Claim \ref{claim_Cayley}} \label{sec_claim}

Our strategy of proving Claim \ref{claim_Cayley} is to transform the question to the positivity of two or one variable polynomials.  The disadvantage is that the degree becomes higher, and the coefficients become large numbers.  It is interesting to see whether one can apply the theory of \emph{positive polynomials} to prove it; see for instance \cites{Schm91, Put93}.  To the best of our knowledge, those positivstellens\"atze are not quite constructive.

Note that both \eqref{fct_Cayley_f} and \eqref{fct_Cayley_g} are quartic polynomials in $s_1$.  We first recall a classical result on the roots of quartic polynomials with real coefficients.

\begin{lem}[\cite{Rees22}] \label{lem_quartic}
    For a quartic polynomial of real coefficients $a x^4 + b x^3 + c x^2 + d x + e$ with $a > 0$, if its discriminant\footnote{$\Dt = b^2 c^2 d^2 - 4 a c^3 d^2 - 4 b^3 d^3 + 18 a b c d^3 - 27 a^2 d^4 - 4 b^2 c^3 e + 16 a c^4 e + 18 b^3 c d e - 80 a b c^2 d e - 6 a b^2 d^2 e + 144 a^2 c d^2 e - 27 b^4 e^2 + 144 a b^2 c e^2 - 128 a^2 c^2 e^2 - 192 a^2 b d e^2 + 256 a^3 e^3$} $\Dt > 0$ and either
    \begin{enumerate}
        \item $8ac - 3b^2 \geq 0$ or
        \item $8ac - 3b^2 < 0$ and $-3 b^4 - 9 a b^4 + 16 a b^2 c + 48 a^2 b^2 c - 64 a^3 c^2 - 
 64 a^2 b d + 256 a^3 e > 0$,
    \end{enumerate}
    then it has no real roots.  In particular, the polynomial has a positive lower bound over $x\in\BR$.
\end{lem}
In \cite{Rees22}, the conditions are written in terms of the reduced form of a quartic polynomial.

\begin{proof}[Proof of Claim \ref{claim_Cayley}]
    \textit{Step 1: boost the positivity of $\ff$}.
    The leading order terms of $\ff(s_1, -\vsm - ws_1, w)$ in $s_1$ and $w$ do not all have good sign.  One can use $\fg(s_1, -\vsm - ws_1, w)$ to help:
    \begin{align} \label{Cayley_quartic} \begin{split}
        &\quad 21\cdot\ff(s_1, -\vsm - ws_1, w) - (6 w^6 + 20 w^4)\cdot\fg(s_1, -\vsm - ws_1, w) \\
        &= \fa(w) s_1^4 + \fb(w,\vsm) s_1^3 + \fc(w,\vsm) s_1^2 + \fd(w,\vsm) s_1 + \fe(w,\vsm) ~,
    \end{split} \end{align}
    where
    {\allowdisplaybreaks
    \begin{align*}
        \fa(w) &= 18 w^8 + 101 w^4 + 147 w^2 + 294 ~, \\
        \fb(w,\vsm) &= 4(3\vsm - 2) w^7 + (- 296\vsm + 377) w^5 - 42(\vsm - 19) w^3 - 21(6\vsm - 49) w ~, \\
        \begin{split} \fc(w,\vsm) &= 162 w^{10} - 108(\vsm - 9) w^8 - 3(2\vsm^2 + 252\vsm - 583) w^6 \\
        &\quad + (- 356\vsm^2 + 591\vsm + 2052) w^4 + 21(9\vsm^2 + 44\vsm + 53) w^2 + 21(\vsm^2 + 41\vsm + 164) ~, \end{split} \\
        \begin{split} \fd(w,\vsm) &= 324(\vsm + 1) w^9 - 9 (24\vsm^2 - 168 \vsm - 163) w^7 - 6 (218\vsm^2 - 296\vsm - 159) w^5 \\
        &\quad + 21(70\vsm^2 + 146\vsm - 93) w^3 + 42(4\vsm^3 + 3\vsm^2 - 17\vsm + 98) w ~, \end{split} \\
        \begin{split} \fe(w,\vsm) &= 162(\vsm + 1)^2 w^8 - 3 (36\vsm^3 - 180\vsm^2 - 381\vsm - 187) w^6 \\
        &\quad + (- 584\vsm^3 + 189\vsm^2 + 168\vsm + 441) w^4 + 21(4\vsm^4 + 56\vsm^3 + 57\vsm^2 - 207\vsm + 32) w^2 \\
        &\quad + 84 (3 + \vsm) (6 + \vsm) (6 - \vsm^2) ~. \end{split}
    \end{align*}
    }
    
    It suffices to show that \eqref{Cayley_quartic} is positive for any $s_1$, $w$ and $\vsm\in[0,\sqrt{6})$.  To apply Lemma \ref{lem_quartic}(i), one needs to verify that $\fa(w) > 0$, $8\cdot\fa(w)\cdot\fc(w,\vsm) - 3(\fb(w,\vsm))^2 \geq 0$ and $\Dt(w,\vsm) > 0$, where $\Dt(w,\vsm)$ is the discriminant of \eqref{Cayley_quartic} as a quartic polynomial in $s_1$. 

    \textit{Step 2: $\fa$ and $8\fa\fc - 3\fb^2$}.
    Obviously, $\fa(w) > 0$.  The function $(8\fa\fc - 3\fb^2)(w,\vsm)$ is a quadratic polynomial in $\vsm$, with coefficients being polynomials in $w$.  Express it in terms of degree $2$ Bernstein polynomials \eqref{Bernstein_poly} for $\vsm\in[0,\frac{5}{2}]\supset[0,\sqrt{6})$:
    \begin{align*} 
        8\cdot\fa(w)\cdot\fc(w,\vsm) - 3(\fb(w,\vsm))^2 &= \fq_0(w)\cdot\fb_{0,2}(\frac{2\vsm}{5}) + \fq_1(w)\cdot\fb_{1,2}(\frac{2\vsm}{5}) + \fq_2(w)\cdot\fb_{2,2}(\frac{2\vsm}{5})
    \end{align*}
    where
    {\allowdisplaybreaks
    \begin{align*}
        \begin{split}
            \fq_0(w) &= 3 \left( 2700096 + 1163799 w^2 + 1330364 w^4 + 1062698 w^6 + 1580412 w^8 \right. \\
            &\qquad \left. + 903159 w^{10} + 429824 w^{12} + 127520 w^{14} + 46656 w^{16} + 7776 w^{18} \right) ~,
        \end{split} \\
        \begin{split}
            \fq_1(w) &= 3 \left(3543876 + 2815344 w^2 + 3011589 w^4 + 1886493 w^6 + 1951477 w^8 \right. \\
            &\qquad \left. + 905359 w^{10} + 411694 w^{12} + 82400 w^{14} + 40176 w^{16} + 7776 w^{18} \right) ~,
        \end{split} \\
        \begin{split}
            \fq_2(w) &= 13471668 + 16035642 w^2 + 10141992 w^4 + 4948839 w^6 + 4735126 w^8 \\
            &\qquad + 1238577 w^{10} + 993492 w^{12} + 103740 w^{14} + 101088 w^{16} + 23328 w^{18} ~.
        \end{split}
    \end{align*}
    }
    It follows that $(8\fa\fc - 3\fb^2)(w,\vsm) > 0$ for any $w\in\BR$ and $\vsm\in[0,\sqrt{6})$.

    \textit{Step 3: the discriminant of \eqref{Cayley_quartic}}.  The discriminant $\Dt(w,\vsm)$ is a degree $12$ polynomial in $\vsm$, whose coefficients are polynomials in $w^2$.  Note that
    \begin{align*}
        \Dt(0,\vsm) &= 2^7\cdot3^6\cdot7^7\cdot(3 + \vsm) (6 + \vsm) (6 - \vsm^2) (2704 + 1352\vsm + 4697\vsm^2 + 2098\vsm^3 + 225\vsm^4)^2
    \end{align*}
    is positive for $\vsm\in[0,\sqrt{6})$.  Similar to step 2, express $\Dt(w,\vsm) - \Dt(0,\vsm)$ in terms of degree $12$ Bernstein polynomials \eqref{Bernstein_poly} for $\vsm\in[0,\frac{5}{2}]\supset[0,\sqrt{6})$:
    \begin{align*}
        \Dt(w,\vsm) - \Dt(0,\vsm) &= w^2 \sum_{j=0}^{12}\fr_j(w^2)\cdot\fb_{j,12}(\frac{2\vsm}{5})
    \end{align*}
    where $\fr_j(w^2)$'s are polynomials in $w^2$.

    For $j\in\{0,1,\ldots,11\}$, $\fr_j(0) > 0$, and the coefficient of $\fr_j(w^2)$ in $w^{2k}$ is non-negative for every $k\in\BN$.  Therefore, $\fr_j(w^2) > 0$ for all $w$.  For instance, up to a factor of $2^4\cdot3^3$, $\fr_0(w^2)$ is
    {\allowdisplaybreaks
    \begin{align*}
        &\quad 1385917274395600128 + 17882051877103563072 w^2 + 34045174584434955312 w^4 \\
        & + 37514728151471928924 w^6 + 45171938822844834756 w^8 + 56967180342756234327 w^{10} \\
        & + 64089597886958676060 w^{12} + 62394851560231359960 w^{14} + 57166154554106018456 w^{16} \\
        & + 50558926859757015672 w^{18} + 41555918902612042992 w^{20} + 31350500674098250736 w^{22} \\
        & + 21809242957716591384 w^{24} + 14077181559906761562 w^{26} + 8343090076003475816 w^{28} \\
        & + 4478601647575856184 w^{30} + 2153781244629264384 w^{32} + 917430601362808380 w^{34} \\
        & + 342862102135296960 w^{36} + 111690528108904404 w^{38} + 31346704386587268 w^{40} \\
        & + 7373775300428007 w^{42} + 1388666535802524 w^{44} + 196789706537040 w^{46} \\
        & + 19301707065096 w^{48} + 1145559339252 w^{50} + 30304891584 w^{52} ~.
    \end{align*}
    }

    For the last coefficient polynomial,
    \begin{align*}
        \fr_{12}(w^2) &= 2^{-14} (\fs_2(w^2) + \fs_1(w^2) + \fs_0(w^2)) ~,
    \end{align*}
    where
    {\allowdisplaybreaks
    \begin{align*}
        \fs_2(w^2) &= 10368 w^{36} (80813044224 w^{16} + 16087971538656 w^{12} \\
        &\qquad\qquad  - 108674110379376 w^8 - 1558904597746992 w^4 + 11809499692540555) ~, \\
        \fs_1(w^2) &= 1296 w^{34} (17336846866176 w^{16} + 223883373358080 w^{12} \\
        &\qquad\qquad  - 3827466501483744 w^8 - 26665451576093568 w^4 + 724970964850265675) ~,
    \end{align*}
    }and $\fs_0(w^2)$ has non-negative coefficients in $w^{2k}$ for every $k\in\BN$.  Up to the factor $2^{-14}$, $\fs_2(w^2) + \fs_1(w^2)$ are the top $10$ degree terms of $\fr_{12}(w^2)$.  The lower order terms, $\fs_0(w^2)$, have good coefficients.
    
    Define two degree $4$ polynomials $\td{\fs}_2(u)$ and $\td{\fs}_1(u)$ by $\fs_2(w^2) = 10368 w^{36}\cdot\td{\fs}_2(w^4)$ and $\fs_1(w^2) = 1296 w^{34}\cdot\td{\fs}_1(w^4)$.  One can check that $\td{\fs}_2(u)$ satisfies the condition of Lemma \ref{lem_quartic} (ii), and hence $\td{\fs}_2(w^4) > 0$ for all $w$.  For $\td{\fs}_1(u)$, its discriminant is negative, and has two real roots and two complex conjugate roots.  By the location of roots theorem, the two real roots of $\td{\fs}_1(u)$ are within $[-20,-10]$ and $[-210,-200]$.  It follows that $\td{\fs}_1(u) > 0$ when $u\geq0$, and $\td{\fs}_1(w^4) > 0$ for all $w$.  This finishes the proof of Claim \ref{claim_Cayley}.
\end{proof}


\end{document}